\documentclass[11pt,reqno]{amsart}
\usepackage{amssymb,amsmath,amsfonts,amsthm,enumerate,stmaryrd,tensor,mathtools,dsfont,upgreek,bbm,mathrsfs,nameref,microtype,xcolor,bm,tikz,stmaryrd,todonotes,hyperref,xfrac}
\usetikzlibrary{arrows}
\usepackage[left=0.891 in, right=0.891 in,top=0.891 in, bottom=0.891 in]{geometry}

\numberwithin{equation}{section}
\definecolor{popblue}{RGB}{55,115,255}
\definecolor{lightbl}{RGB}{155,205,255}
\definecolor{depthbl}{RGB}{145,215,255}
\definecolor{fancyre}{RGB}{225,55,115}
\definecolor{darkblu}{RGB}{15,75,185}
\definecolor{mellowy}{RGB}{225,225,35}

\renewcommand{\tilde}[1]{\widetilde{#1}}
\renewcommand{\Bar}{\overline}

\newcommand{\R}{\mathbb{R}}
\newcommand{\N}{\mathbb{N}}
\newcommand{\Z}{\mathbb{Z}}

\newcommand{\C}{\mathbb{C}}

\renewcommand{\P}{\mathbb{P}}
\newcommand{\T}{\mathbb{T}}

\newcommand{\m}{\mathrm}

\newcommand{\lv}{\lVert}
\newcommand{\rv}{\rVert}

\newcommand{\al}{\alpha}
\newcommand{\be}{\beta}
\newcommand{\es}{\varnothing}

\newcommand{\ep}{\varepsilon}
\newcommand{\f}{\frac}

\newcommand{\gam}{\gamma}
\newcommand{\del}{\delta}

\newcommand{\pd}{\partial}

\newcommand{\grad}{\nabla}
\newcommand{\bpm}{\begin{pmatrix}}
\newcommand{\epm}{\end{pmatrix}}

\newcommand{\emb}{\hookrightarrow}

\renewcommand{\le}{\leqslant}
\renewcommand{\ge}{\geqslant}

\newcommand{\tnorm}[1]{\lv#1\rv}

\newcommand{\p}[1]{\left(#1\right)}
\newcommand{\bp}[1]{\Big(#1\Big)}

\newcommand{\tp}[1]{(#1)}

\newcommand{\babs}[1]{\Big|#1\Big|}

\newcommand{\tabs}[1]{|#1|}

\newcommand{\bsb}[1]{\Big[{#1}\Big]}

\newcommand{\tsb}[1]{[{#1}]}

\newcommand{\bcb}[1]{\Big\{{#1}\Big\}}
\newcommand{\tcb}[1]{\{{#1}\}}

\newcommand{\z}[1]{\mathring{#1}}

\renewcommand{\bf}[1]{\mathbf{#1}}
\newcommand{\ii}{\mathsf{i}}

\newtheorem{prop}{\color{popblue}{Proposition}}[section]

\newtheorem{lem}[prop]{\color{popblue}{Lemma}}

\newtheorem{innercustomthm}{\color{popblue}{Theorem}}
\newenvironment{customthm}[1]
{\renewcommand\theinnercustomthm{#1}\innercustomthm}
{\endinnercustomthm}

\author{Huy Q. Nguyen}
\address{
	Department of Mathematics\\
	University of Maryland\\
	College Park, MD 20742, USA
}
\email[H. Q. Nguyen]{hnguye90@umd.edu}
\thanks{H. Q. Nguyen was supported by an NSF Grant (DMS \#2205710).}
\author{Noah Stevenson}
\address{
	Department of Mathematics\\
	Princeton University\\
	Princeton, NJ 08544, USA
}
\email[N. Stevenson]{stevenson@princeton.edu}
\thanks{N. Stevenson was supported by an NSF Graduate Research Fellowship}

\DeclareFontFamily{U}{cbgreek}{}
\DeclareFontShape{U}{cbgreek}{m}{n}{
  <-6>    grmn0500
  <6-7>   grmn0600
  <7-8>   grmn0700
  <8-9>   grmn0800
  <9-10>  grmn0900
  <10-12> grmn1000
  <12-17> grmn1200
  <17->   grmn1728
}{}
\DeclareFontShape{U}{cbgreek}{bx}{n}{
  <-6>    grxn0500
  <6-7>   grxn0600
  <7-8>   grxn0700
  <8-9>   grxn0800
  <9-10>  grxn0900
  <10-12> grxn1000
  <12-17> grxn1200
  <17->   grxn1728
}{}

\makeatletter
\newcommand{\normalorbold}{%
  \ifnum\pdf@strcmp{\math@version}{bold}=\z@ bx\else m\fi
}
\makeatother

\title[Large traveling waves for Darcy flow]{On large periodic traveling surface waves in porous media}



\begin{document}
\begin{abstract}
    We study large traveling surface waves within a two-dimensional finite depth, free boundary, homogeneous, incompressible and viscous fluid governed by Darcy's law. The fluid is bound by a gravitational force to a flat rigid bottom and meets an atmosphere of constant pressure at the top with its free surface, where it does not experience any capillarity effects. Additionally, the fluid is subject to a fixed, but arbitrarily selected, forcing data profile with variable amplitude. We use the Riemann mapping to equivalently reformulate the resulting two-dimensional free boundary problem as a single one-dimensional fully nonlinear pseudodifferential equation for a function describing the domain's geometry. By discovering a hidden ellipticity in the reformulated equation, we are able to import a global implicit function theorem to  construct a connected set of traveling waves, containing both the quiescent solution and large amplitude members. We find that either solutions continue to exist for arbitrarily large data amplitude or else one of a finite number of meaningful breakdown scenarios must occur. This work stands as the first non perturbative construction of large traveling surface waves in any free boundary viscous fluid without surface tension.
\end{abstract}
\maketitle
\section{Introduction}\label{section on introduction}

\subsection{Equations of motion and traveling ansatz}\label{subsection on equations of motion and traveling ansatz}

In this paper we study a two-dimensional finite depth layer of fluid existing within a porous medium. The fluid is bounded below by an impermeable, flat, and rigid bottom and, at the top, meets an atmosphere of constant pressure with a free boundary.

Let us denote the dynamically evolving fluid domains by $t\mapsto\Omega(t)\subset\R^2$. At each time $t\ge0$ the boundary of $\Omega(t)$ is assumed to have two connected components: $\pd\Omega(t) = \Sigma(t)\sqcup\Sigma_0$, with $\Sigma_0 = \R\times\tcb{0}$ denoting the rigid bottom and $t\mapsto\Sigma(t)\subset\R\times(0,\infty)$ denoting the time-evolving free boundary. 

The Eulerian description of the fluid motion consists of two unknowns: the vector velocity $V(t,\cdot):\Omega(t)\to\R^2$ and the scalar pressure $P(t,\cdot):\Omega(t)\to\R$. In addition to the gravitational force with acceleration $g\ge0$ in the $-e_2$ direction, the fluid is permitted to experience a known vector applied force $F(t,\cdot):\Omega(t)\to\R^2$ in the bulk and a known scalar applied pressure $\Phi(t,\cdot):\Sigma(t)\to\R$ on the surface. The equations in the bulk relating these functions are written as follows:
\begin{equation}\label{ID FB}
    \grad\cdot V(t,\cdot) = 0\text{ and }(\mu/\iota)V(t,\cdot) + \grad P(t,\cdot) = -\rho g e_2 + F(t,\cdot)\text{ in }\Omega(t).
\end{equation}

The first identity in~\eqref{ID FB} is the constraint of incompressibility. The second identity above is known as Darcy's law~\cite{DARCY}. This is an empirically observed relationship between the velocity and the pressure of a fluid flowing through a porous medium, and can be rigorously derived from Stokes equations via homogenization  \cite{MR1020348}. The constants $\mu,\iota,\rho\in\R^+$ are the dynamic viscosity, the quotient of permeability and effective porosity, and the density, respectively.

The bulk equations~\eqref{ID FB} are coupled to a series of boundary conditions. On the bottom we simply have the no penetration condition for the velocity,
\begin{equation}\label{ID FB BBC}
    e_2\cdot V(t,\cdot) = 0\text{ on }\Sigma_0,
\end{equation}
while on the top free surface we have the kinematic condition and balance of pressure:
\begin{equation}\label{ID FB TBC}
    \mathcal{V}\tp{\Sigma(t)}\tp{\cdot} = V(t,\cdot)\cdot n(t,\cdot)\text{ and }P(t,\cdot) = P_{\m{ext}} + \Phi(t,\cdot)\text{ on }\Sigma(t,\cdot).
\end{equation}
The object $\mathcal{V}\tp{\Sigma(t,\cdot)}$ is the scalar normal speed of the free surface $\Sigma(t)$. The outward pointing unit normal vector field is $n(t,\cdot):\Sigma(t)\to\R^2$. Finally,  $P_{\m{ext}}\in\R$ is the fixed constant ambient exterior pressure.

The coupled system of equations~\eqref{ID FB}, \eqref{ID FB BBC}, and~\eqref{ID FB TBC} comprise what is known as the one-phase Muskat problem~\cite{MUSKAT}. For more information on Darcy's law and flow through porous media, the interested reader is referred to the books of Bear~\cite{BEAR,MR3752178}.

Our interest in this work regards the construction and analysis of traveling wave solutions to the above equations. By a traveling wave, we mean a solution to~\eqref{ID FB}, \eqref{ID FB BBC}, and~\eqref{ID FB TBC} that is stationary when viewed in a system of coordinates moving at a constant velocity $\gam$ in the $e_1$ direction. More precisely, we make the ansatz that there exists a stationary domain $\Omega\subset\R^2$, with boundaries $\Sigma_0,\Sigma$ satisfying $\pd\Omega = \Sigma_0\sqcup\Sigma$, such that
\begin{equation}\label{stationary unknown domains}
    \Omega(t) = \tcb{x + t\gam e_1\;:\;x\in\Omega},\;\Sigma(t)  = \tcb{x + t\gam e_1\;:x\in\Sigma},\text{ and }\Sigma_0 = \R\times\tcb{0}\text{ for all }t\ge0.
\end{equation}

On this stationary domain we have the time-independent traveling velocity and pressure unknowns $v:\Omega\to\R^2$ and $p:\Omega\to\R$ that are related to $V$ and $P$ via:
\begin{equation}\label{stationary functional unknowns}
    V(t,x) = v(x - t\gam e_1)\text{ and }P(t,x) = p(x - t\gam e_1)\text{ for all }t\ge0,\;x\in\Omega(t).
\end{equation}
Of course, we also make the ansatz that the supplemental forcing to the system is stationary in the traveling frame as well. More precisely, we suppose that $f:\Omega\to\R^2$ and $\phi:\Sigma\to\R$ are given functions that satisfy
\begin{equation}
    F(t,x) = f(x - t\gam e_1)\text{ and }\Phi(t,x) = \phi(x - t\gam e_1)
\end{equation}
for all $t\ge0$ and appropriate $x$.

The traveling unknowns $\Omega$, $v$, and $p$ are thus related to the given $f$ and $\phi$ via the traveling wave equations for free boundary incompressible Darcy flow:
\begin{equation}\label{traveling FBIDF}
    \begin{cases}
        \grad\cdot v = 0\text{ and }\tp{\mu/\iota}v + \grad p = -\rho g e_2 + f&\text{ in }\Omega,\\
        e_2\cdot v = 0&\text{on }\Sigma_0,\\
        \tp{v - \gam e_1}\cdot\nu = 0\text{ and }p = P_{\m{ext}} + \phi&\text{on }\Sigma,
    \end{cases}
\end{equation}
where $\nu:\Sigma\to\R^2$ denotes the outward pointing unit normal. We emphasize that while the above formulation is time-independent, the domain $\Omega$ is still one of the unknowns of the problem.

Our goal in this work is  to construct and analyze large amplitude solutions to system~\eqref{traveling FBIDF}. We shall think of the supplemental forcing $f$ and $\phi$ as having an arbitrary but fixed profile $\mathsf{f}$ and $\upphi$ (not both identically zero), but tunable scalar strength $\kappa\ge0$ so that $f = \kappa\mathsf{f}$ and $\phi = \kappa\upphi$. When $\kappa = 0$ the system~\eqref{traveling FBIDF} admits a unique quiescent and flat equilibrium solution. By an implicit function theorem argument, \cite{MR4690615} proved that for small $\kappa>0$ there is a locally unique curve of nontrivial and small amplitude solutions $\kappa\mapsto\tp{\Omega_\kappa,v_\kappa,p_\kappa}$. We shall use global continuation based on the Leray-Schauder degree to extend this local mapping to a global connected set of solutions. In doing so, we will obtain large traveling waves that are not close to any equilibrium; moreover, along the connected set of solutions we prove that either arbitrarily large waves continue to exist or a else a geometrically meaningful quantity related to the domain $\Omega$ is becoming arbitrarily large. 

Despite the superficially simple appearance of the system of equations~\eqref{traveling FBIDF}, this task of constructing large solutions is not a  trivial undertaking, even if we ignore the immediate complication that the domain itself is one of the unknowns. Note that there are not any effects of surface tension and so we do not have available the extra two derivative regularization present within the corresponding system with positive capillarity. In fact, in a certain sense to be elucidated later, system~\eqref{traveling FBIDF} is not even unconditionally `elliptic'. As a consequence, we require that certain inequalities (see~\eqref{conditions for the hidden ellipticity}) relating the gravitational acceleration, the wave speed, and the forcing data $f$ and $\phi$ be satisfied in order to crucially exploit a `hidden ellipticity' upon which our construction hinges.

\subsection{Literature Survey}\label{subsection on literature survey}

The literature regarding the dynamical problem~\eqref{ID FB}, \eqref{ID FB BBC}, and~\eqref{ID FB TBC} and its closely related variants, e.g. in two or three spatial dimension, with or without surface tension, consisting of one or multiple layers, etc. is diverse, vast, and witnessing recent flourish. Accordingly, since our focus is specifically on traveling waves, we make no attempt here at a comprehensive review of the full dynamic problem. Rather, we only mention the following results on local-in-time well-posedness~\cite{MR2318314,MR2313156,MR3171344,MR2753607,MR3071395,MR3415681,MR3861893,MR4131404,MR4097324,MR4242131,MR4541917,MR4227171,MR4090462}, on global-in-time well-posedness nearby equilibria~\cite{MR2998834,MR3595492,MR3899970,MR3869383,GWPMSSC,MR4363243,MR4487512,MR4348695}, on global-in-time well-posedness for large data~\cite{MR4655356,DGNGWP3D,MR4581109}, and on instabilities and singularity formation~\cite{MR2993754,MR3215843,MR3048596,MR3519969}; see also the survey articles~\cite{MR3608884,MR4064479}.

The porous medium traveling wave problem~\eqref{traveling FBIDF} is of course superficially related to the classical and modern study of traveling wave solutions to the free boundary incompressible Euler equations. This body of literature, that is also called the traveling water wave problem, possesses formidable breadth and depth as it dates back to the observations and conjecture of Stokes~\cite{STOKES} in the mid-nineteenth century. Rather than attempt a summary here, we refer the reader to the survey articles~\cite{Toland_1996,MR1984157,Groves_2004,Strauss_2010,MR4406719}.

One of the key distinctions between the aforementioned setting and our own~\eqref{traveling FBIDF} is that we consider a \emph{viscous} fluid. In stark contrast with the inviscid case, the rigorous mathematical study of the viscous traveling surface wave problem is a much more recent development, beginning with the seminal work of Leoni and Tice~\cite{MR4630597}. They constructed small amplitude solitary traveling wave solutions to the free boundary incompressible Navier-Stokes equations. Subsequent analyses on variations of the Navier-Stokes set up arrived with the work of Stevenson and Tice~\cite{MR4337506,COMPRESS,MR4787851,BORE} and Koganemaru and Tice~\cite{MR4609068,MR4785303}. Aspects of the viscous traveling surface wave problem in the context of the shallow water equations were explored by Stevenson and Tice~\cite{VCV,SWB} and Stevenson~\cite{MR4873829}. For free boundary Darcy flow, small amplitude traveling waves were first constructed by Nguyen and Tice~\cite{MR4690615}. Subsequently, Nguyen~\cite{NGCD} produced the first large amplitude traveling waves for Darcy flow with the effects of surface tension taken into account. Finally, Brownfield and Nguyen~\cite{MR4797733} constructed slowly traveling waves nearby large stationary Darcy flow equilibria. Therefore, in light of these former results, our current work stands as the first non-perturbative construction of large amplitude traveling wave solutions to the equations of free boundary incompressible Darcy flow in the absence of surface tension.

\subsection{Nondimensional, conformal, and nonlocal reformulations}\label{subsection on conformal and nonlocal reformulation}

We now endeavor to reformulate system~\eqref{traveling FBIDF} into an equivalent manifestation of the problem that is more convenient for the subsequent analysis. This is also a necessary step for us to take to properly state the main results of this paper in the forthcoming Section~\ref{subsection on main results and discussion}. We proceed through several steps.

Our initial reformulation of the equations is more cosmetic. It serves to eliminate one of the unknowns, namely the velocity, and to nondimensionalize the system. We are interested in solutions to~\eqref{traveling FBIDF} that are periodic in the $e_1$ direction and we select our characteristic length scale $\m{L}>0$ to be a reference period length of the solution (say the minimal period length of the known data $(f,\phi)$). The characteristic velocity scale $\m{U}>0$ is then defined in terms of $\m{L}$ and the other physical parameters via $\m{U} = \mu\m{L}/\iota\rho$. We define the nondimensional unknown domain $\Upomega$ and pressure $\Pi:\Upomega\to\R$ via 
\begin{equation}\label{nondimensional domain and pressure}
    \Omega = \tcb{x\in\R^2\;:\;x/\m{L}\in\Upomega}\text{ and }\rho\m{U}^2\Pi(x/\m{L}) = p(x) - P_{\m{ext}} + g\tp{e_2\cdot x}\text{ for }x\in\Omega.
\end{equation}
The inverse Froude number $\m{g}$, nondimensional wave speed $\m{c}$, and nondimensional forcing terms $\mathsf{f}$ and $\upphi$ are
\begin{equation}\label{nondimensional numbers and forcing}
    \m{g} = \rho g\iota/\mu,\;\m{c} = \gam \rho\iota/\mu\m{L},\;f(x) = \kappa\tp{\rho\m{U}^2/\m{L}}\mathsf{f}\tp{x/\m{L}},\text{ and }\phi(x) = \kappa\rho\m{U}^2\upphi(x/\m{L}).
\end{equation}
Note that in the above we have introduced a nondimensional forcing strength parameter $\kappa\ge0$.

By taking the divergence of the second equation of~\eqref{traveling FBIDF}, we compute that the equations to be satisfied by $\Upomega$, $\Pi$, and the nondimensional forcing terms are
\begin{equation}\label{nondimensional pressure only form of Darcy flow}
\begin{cases}
    \Delta\Pi = \kappa\grad\cdot\mathsf{f}&\text{in }\Upomega,\\
    e_2\cdot\grad\Pi = \kappa\tp{e_2\cdot \mathsf{f}}&\text{on }\Upsigma_0,\\
    \tp{-\grad\Pi + \kappa\mathsf{f} - \m{c}e_1}\cdot\upnu = 0\text{ and }\Pi = \m{g}\tp{e_2\cdot\m{id}_{\R^2}} + \kappa\upphi&\text{on }\Upsigma,
\end{cases}
\end{equation}
where we have decomposed $\pd\Upomega = \Upsigma\sqcup\Upsigma_0$ with $\Upsigma_0 = \R\times\tcb{0}$ and denoted the outward pointing unit normal on $\Upsigma$ by $\upnu:\Upsigma\to\R^2$.

The dimensional velocity $v:\Omega\to\R^2$ appearing in~\eqref{traveling FBIDF} can be recovered from the solution to~\eqref{nondimensional pressure only form of Darcy flow} via 
\begin{equation}
    v(x) = f(x) - \rho ge_2 - \grad p(x)\text{ for }x\in\Omega.
\end{equation}
with $f$ defined in terms of $\mathsf{f}$ as in~\eqref{nondimensional numbers and forcing} and $p$ defined in terms of $\Pi$ as in~\eqref{nondimensional domain and pressure}.

As we have previously alluded, central difficulty one encounters with systems~\eqref{traveling FBIDF} and~\eqref{nondimensional pressure only form of Darcy flow} is that the fluid domains $\Omega$ and $\Upomega$ themselves are among the unknowns of the problem. Our next reformulation of the system of equations is meant to take a first step in addressing this fact. We shall, in a procedure known as \emph{flattening}, devise a change of coordinates depending on the unknown domain $\Upomega$ itself that allows us to pull the system of equations back to a fixed flat slab domain. The differential operators are complicated by this process. In fact it will be revealed that system~\eqref{nondimensional pressure only form of Darcy flow} is fully nonlinear.

Since we are working in two spatial dimensions, it is convenient to chart the configuration space of possible domains with the help of the Riemann mapping. We shall fix a nondimensional conformal depth parameter $h>0$ and make the ansatz that there exists a biholomorphic mapping $R:\R\times\tp{0,h}\to\Upomega$, extending to a diffeomorphism of the closures, such that
\begin{equation}
    R(x + e_1) = R(x) + e_1\text{ for all }x\in\R\times\tp{0,h}\text{ and }R(\R\times\tcb{0})=\R\times\tcb{0}.
\end{equation}
Thanks to the Cauchy-Riemann equations, such a mapping $R$ is essentially entirely determined by the trace of its second component on the top boundary $\R\times\tcb{h}$. Under suitable regularity assumptions on $\Upomega$, there exists 
\begin{equation}\label{Riemann mapping decomposition}
    \psi\in C^{1,\al}\tp{\T}\text{ satisfying }\int_{\T}\psi = 0\text{ such that }R(w,z) = (w,z) + \tp{\mathcal{E}\psi}(w,z),\;\tp{w,z}\in\R\times\tp{0,h},
\end{equation}
where $0<\al<1$ and $\T = \R/\Z$ denotes the one-dimensional unit diameter torus. The operator $\mathcal{E}$ in~\eqref{Riemann mapping decomposition} is the Cauchy-Riemann solver, defined in terms of the horizontal Fourier transform $\mathscr{F}$ via
\begin{equation}\label{definition of the Cauchy Riemann solver}
    \mathscr{F}\tsb{\mathcal{E}\psi}\tp{\xi,z} = \mathds{1}_{\Z\setminus\tcb{0}}\tp{\xi}\bp{-\ii\f{\xi\cosh\tp{2\pi|\xi|z}}{|\xi|\sinh\tp{2\pi|\xi|h}}e_1 + \f{\sinh\tp{2\pi|\xi|z}}{\sinh\tp{2\pi|\xi|h}}e_2}\mathscr{F}[\psi]\tp{\xi}
\end{equation}
for $\xi\in\Z$ and $z\in\tp{0,h}$. In particular, we have $(\mathcal{E} \psi)_2(\cdot, h)=\psi$. The operator $\mathcal{E}$ is studied in greater detail in Lemma~\ref{lemma on bounds on the Cauchy Riemann solver} below.

The map $R$ discussed above permits us to change coordinates in~\eqref{nondimensional pressure only form of Darcy flow} and reformulate to an equivalent system on the fixed periodic domain $\T\times\tp{0,h}$ for the unknown function $\psi$ as in~\eqref{Riemann mapping decomposition} and the flattened pressure $q:\T\times\tp{0,h}\to\R$ defined by $q = - \m{g}h + \Pi\circ R$. This system for $\tp{\psi,q}$ reads:
\begin{equation}\label{flattened formulation of the darcy flow}
    \begin{cases}
        \Delta q = \frac{\kappa}{2}\tabs{\grad R}^2\tp{\grad\cdot\mathsf{f}}\circ R&\text{in }\T\times\tp{0,h},\\
        \pd_2q = \kappa\tp{\mathsf{f}\circ R}\cdot\tp{\pd_1R}^\perp&\text{on }\T\times\tcb{0},\\
        \pd_2q = \m{c}\pd_1\psi + \kappa\tp{\mathsf{f}\circ R}\cdot\tp{\pd_1 R}^\perp\text{ and }q = \m{g}\psi + \kappa\tp{\upphi\circ R}&\text{on }\T\times\tcb{h},\\
        R(w,z) = (w,z) + \tp{\mathcal{E}\psi}\tp{w,z}&\text{for all }\tp{w,z}\in\T\times\tp{0,h}. 
    \end{cases}
\end{equation}

To conclude this subsection, we now need to discuss one final equivalent reformulation of the traveling wave problem. Remarkably, we are able to reduce~\eqref{flattened formulation of the darcy flow} into a single fully nonlinear and nonlocal pseudodifferential equation for the domain coordinate function $\psi$. The description of this reduction procedure requires the introduction of two additional linear operators. These are the mappings
\begin{equation}\label{help_solve_1}
    S:C^\al\tp{\T\times(0,h)}\times C^{1,\al}\tp{\T}\to C^{1,\al}\tp{\T}\text{ and }G:C^{2,\al}\tp{\T}\to C^{1,\al}\tp{\T}
\end{equation}
with actions
\begin{equation}\label{help_solve_2}
    S(f^1,f^2) = \m{Tr}_{\T\times\tcb{h}}\pd_2u\text{ and }Gf^3 = \m{Tr}_{\T\times\tcb{h}}\pd_2v
\end{equation}
where $u$ and $v$ are the unique solutions to the PDEs
\begin{equation}\label{help_solve_3}
    \begin{cases}
        \Delta u = f^1&\text{in }\T\times\tp{0,h},\\
        u = 0&\text{on }\T\times\tcb{h},\\
        \pd_2u = f^2&\text{on }\T\times\tcb{0},
    \end{cases}
    \text{ and }
    \begin{cases}
        \Delta v =0&\text{in }\T\times\tp{0,h},\\
        v = f^3&\text{on }\T\times\tcb{h},\\
        \pd_2v = 0&\text{on }\T\times\tcb{0},
    \end{cases}
\end{equation}
for given $\tp{f^1,f^2}\in C^\al\tp{\T\times\tp{0,h}}\times C^{1,\al}\tp{\T}$ and $f^3\in C^{2,\al}\tp{\T}$. Here $\m{Tr}_{\T\times\tcb{z}}$ denotes the trace operator onto $\T\times\tcb{z}$ for $z\in\tcb{0,h}$. Notice that $G$ is the flat, finite-depth Dirichlet-to-Neumann operator
\begin{equation}\label{form:DN}
    \mathscr{F}[Gf]\tp{\xi} = 2\pi\tabs{\xi}\tanh\tp{2\pi h\tabs{\xi}}\mathscr{F}[f]\tp{\xi}.
\end{equation}
By using $S$ and $G$ as defined above, system~\eqref{flattened formulation of the darcy flow} is seen to be equivalent to solving first for $\psi$ in the identity
\begin{multline}\label{reformulation for the coordinate function psi}
    G\tp{\m{g}\psi + \kappa\m{Tr}_{\T\times\tcb{h}}}\tp{\upphi\circ R} - \m{c}\pd_1\psi - \kappa\m{Tr}_{\T\times\tcb{h}}\tp{\mathsf{f}\circ R}\cdot\tp{\pd_1 R}^\perp \\
    + \kappa S\tp{\tp{|\grad R|^2/2}\tp{\grad\cdot\mathsf{f}}\circ R,\m{Tr}_{\T\times\tcb{0}}\tp{\mathsf{f}\circ R}\cdot\tp{\pd_1 R}^\perp} = 0,
\end{multline}
with the understanding that $R$ is defined in terms of $\psi$ as in the final equation of~\eqref{flattened formulation of the darcy flow}, and then solving for $q$ in terms of $\psi$ and the data using the first, second, and fourth equations of~\eqref{flattened formulation of the darcy flow}. The latter task is inverting a simple linear elliptic PDE, and so the all of the remaining difficulty is in the analysis and construction of solutions to equation~\eqref{reformulation for the coordinate function psi}. 

In order for the correspondence between the nonlinear pseudodifferential equation~\eqref{reformulation for the coordinate function psi} and the free boundary system~\eqref{traveling FBIDF} to be invertible we require that the flattening map $R$, defined in terms of $\psi$ via~\eqref{Riemann mapping decomposition}, is a diffeomorphism up to the boundary. However, the operator comprising the left hand side of the equation~\eqref{reformulation for the coordinate function psi} continues to be well-defined for all $\psi\in C^{1,\al}\tp{\T}$ with vanishing mean.

\subsection{Main results and discussion}\label{subsection on main results and discussion}

Let us fix $\al\in(0,1)$ and define the open set of admissible domain coordinate functions
\begin{equation}\label{the set of admissible domain coordinate functions}
    \mathcal{U} = \bcb{\psi\in C^{1,\al}\tp{\T}\;:\;\int_{\T}\psi = 0,\;R_\psi:\T\times\tsb{0,h}\to\T\times[0,\infty)\text{ is injective and }\min_{\T\times[0,h]}\det\grad R_\psi>0},
\end{equation}
where for $\psi\in C^{1,\al}\tp{\T}$ the function $R_\psi$ is the associated Riemann mapping
\begin{equation}\label{The Coordinate Maps}
    R_\psi(w,z) = (w,z) + \tp{\mathcal{E}\psi}\tp{w,z}\text{ for all }\tp{w,z}\in\T\times[0,h]
\end{equation}
and $\mathcal{E}$ is the Cauchy-Riemann solver defined in~\eqref{definition of the Cauchy Riemann solver}. For each $\psi\in\mathcal{U}$ the map $R_\psi$, being biholomorphic, is a smooth diffeomorphism of $\T\times\tp{0,h}$ onto its image that extends to a $C^{1,\al}$-diffeomorphism between the closures. For each such $\psi$, the map $R_\psi$ satisfies $e_2\cdot R_\psi = 0$ when restricted to the set $\T\times\tcb{0}$. Thus the bottom boundary of the image $R_\psi\tp{\T\times\tp{0,h}}$ is always $\T\times\tcb{0}$ while the other boundary component is always a $C^{1,\al}$-curve sitting above.

On the set $\mathcal{U}$ we define the maximal distortion function
\begin{equation}\label{the distortion}
    \mathscr{D}:\mathcal{U}\to(0,\infty)\text{ via }\mathscr{D}\tp{\psi} = \sup\bcb{\f{\m{dist}_{\T}\tp{w_0,w_1}}{\m{dist}_{\T\times\R}\tp{R_\psi(w_0,h),R_\psi\tp{w_1,h}}}\;:\;w_0,w_1\in\T\text{ and }w_0\neq w_1}.
\end{equation}
Roughly speaking, $\mathscr{D}\tp{\psi}$ is a measurement of the Riemann mapping diffeomorphism's complexity in the sense that $1/\mathscr{D}\tp{\psi}$ being smaller means that $\psi$ is closer to the complement of $\mathcal{U}$. Notice that the distortion $\mathscr{D}\tp{\psi}$ as defined above only depends on the trace of the flattening map $R_\psi$ on the top boundary $\T\times\tcb{h}$. Thus, since
\begin{equation}\label{parametrization of the top boundary component of the image}
    \T\ni w\mapsto R_\psi\tp{w,h}\in\R^2
\end{equation}
is a conformal parameterization of the top boundary component of $R_\psi\tp{\T\times(0,h)}$, we can interpret $\mathscr{D}\tp{\psi}<\infty$ as a quantitative verification that~\eqref{parametrization of the top boundary component of the image} is injective and has nowhere vanishing derivative. See Proposition~\ref{prop on the set of admissible coordinate functions} below for more information on the set $\mathcal{U}$ and the function $\mathscr{D}$.

We are now in a position to state our main results, whose proofs can be found in Section~\ref{subsection on synthesis} below. The first main theorem discusses the large solutions to the reformulated problem~\eqref{reformulation for the coordinate function psi} that we are able to construct.

\begin{customthm}{1}[Large traveling wave solutions to the reformulated equation]\label{1_MAIN_THM}
    Let $\mathsf{f}\in C^4\tp{\T\times\R;\R^2}$, $\upphi\in C^4\tp{\T\times\R}$, and $\m{g},\m{c}\in[0,\infty)$ satisfy one of the following ellipticity conditions
    \begin{equation}\label{conditions for the hidden ellipticity}
        \tsb{\m{g}>0\text{ and }\pd_2\upphi\ge e_2\cdot\mathsf{f}}\text{ or }\tsb{\m{c}\ne 0\text{ and }\m{c}(\pd_1\upphi-e_1\cdot\mathsf{f})\ge 0}\quad\text{on~}\T\times \R.
    \end{equation}
    Then, there exists a connected solution set $\tp{0,0}\in\mathcal{C}\subset\R\times\z{C}^{1,\al}\tp{\T}$ (the overset circle denotes the average zero subspace) such that for all $\tp{\upkappa,\psi}\in\mathcal{C}$ the equation~\eqref{reformulation for the coordinate function psi} is satisfied with $\kappa = \upkappa^2$. Moreover, we have that $\mathcal{C}\cap\tp{\R\times\mathcal{U}}\supset\tcb{\tp{0,0}}$ and for all $0<\ep\le1/2$ the function
    \begin{equation}\label{the blow up norm with distortion 1}
        \mathcal{C}\cap\tp{\R\times\mathcal{U}}\ni\tp{\upkappa,\psi}\mapsto|\upkappa| + \tnorm{\psi}_{H^{1/2+\ep}\tp{\T}} + \mathscr{D}\tp{\psi}\in\tp{0,\infty}
    \end{equation}
    is unbounded, where $\mathscr{D}$ is the distortion as in~\eqref{the distortion} and $\mathcal{U}$ is as in~\eqref{the set of admissible domain coordinate functions}. In the absence of the bulk force, i.e. $\mathsf{f} = 0$,  the blow-up quantity~\eqref{the blow up norm with distortion 1} can be improved in the sense that for all $0<\ep\le1/2$ the function 
    \begin{equation}\label{the blow up norm with distortion 2}
        \mathcal{C}\cap\tp{\R\times\mathcal{U}}\ni\tp{\upkappa,\psi}\mapsto|\upkappa| + \tnorm{\psi}_{C^{\ep}\tp{\T}} + \mathscr{D}\tp{\psi}\in\tp{0,\infty}
    \end{equation}
    is unbounded.
\end{customthm}

Theorem~\ref{1_MAIN_THM} tells us that the zero solution $\tp{\kappa,\psi} = 0$ to equation~\eqref{reformulation for the coordinate function psi} is part of a connected set of solutions to the same equation in which the blow-up quantity~\eqref{the blow up norm with distortion 1} (or~\eqref{the blow up norm with distortion 2} when $\mathsf{f} = 0$) grows arbitrarily large. This means that one of three things has to happen along this solution set. It could be the case that the forcing amplitude parameter $\kappa$ becomes arbitrarily large and yet solutions continue to exist. Otherwise, it is some measurement of the geometry of the domain $\Omega_\psi = R_\psi\tp{\T\times\tp{0,h}}\subset\T\times\R$ that is becoming unboundedly large along this solution set: either some norm measuring the regularity of the free boundary's conformal parametrization~\eqref{parametrization of the top boundary component of the image} (meaning $\tnorm{\psi}_X\to\infty$ for $X=H^{1/2 +\ep}\tp{\T}$ or $X = C^\ep\tp{\T}$) or else the distortion (meaning the conformal parameterization is degenerating due to $\mathscr{D}\tp{\psi}\to\infty$). Finally, we note that the first ellipticity condition in \eqref{conditions for the hidden ellipticity} is satisfied if $\m{g}>0$ and $\upphi(x)=\upphi(x_1)$, which is precisely the setup considered in \cite{NGCD} and \cite{MR4797733}. 

Our second main result simply checks that the reduction computations of Section~\ref{subsection on conformal and nonlocal reformulation} are valid along the solution set produced by Theorem~\ref{1_MAIN_THM}. Thus the objects produced by our first main theorem genuinely correspond to large solutions to the traveling wave equations of free boundary incompressible Darcy flow.

\begin{customthm}{2}[Transfer of solutions to the original free boundary problem]\label{2_MAIN_THM}
    Assume the hypotheses of Theorem~\ref{1_MAIN_THM} and let $\mathcal{C}$ be the connected solution set produced therein. Then, for all $\tp{\upkappa,\psi}\in\mathcal{C}\cap\tp{\R\times\mathcal{U}}$ the following hold.
    \begin{enumerate}
        \item The flattening map $R_\psi:\T\times[0,h]\to\T\times[0,\infty)$ is a class $C^{1,\al}$ diffeomorphism onto its image that is smooth in the interior; let us denote $\Omega_\psi = R_\psi\tp{\T\times\tp{0,h}}$.
        \item There exists $\m{v}\in C^\al\tp{\Omega_\psi;\R^2}$ and $\m{p}\in C^{1,\al}\tp{\Omega_\psi}$, class $C^4$ in the interior $\Omega_\psi$, solving in the classical sense the nondimensional equations of traveling free boundary incompressible Darcy flow:
        \begin{equation}\label{traveling_FBIDF}
            \begin{cases}
                \grad\cdot\m{v} = 0\text{ and }\m{v} + \grad\m{p} = -\m{g}e_2 + \upkappa^2\mathsf{f}&\text{ in }\Omega_\psi,\\
                e_2\cdot\m{v} = 0&\text{on }\T\times\tcb{0}\\
                \tp{\m{v} - \m{c}e_1}\cdot\nu = 0\text{ and }\m{p} = \upkappa^2\upphi&\text{on }R_\psi\tp{\T\times\tcb{h}},
            \end{cases}
        \end{equation}
        with $\nu:R_\psi\tp{\T\times\tcb{h}}\to\R^2$ denoting the outward pointing unit normal to the top boundary component of $\Omega_\psi$.
    \end{enumerate}
\end{customthm}

We say that the domains $\Omega_\psi$ discussed in Theorem~\ref{2_MAIN_THM} are graphical if
\begin{equation}
    \exists\;\eta\in C^0\tp{\T}\text{ such that }\Omega_\psi = \tcb{\tp{w,z}\in\T\times[0,\infty)\;:\;0<z<\eta(w)}.
\end{equation}
This condition may fail for sufficiently large traveling waves. Indeed, $\Omega_\psi$ is not graphical for $\psi\in\mathcal{U}$ only if there exists a point $w\in\T$ such that we have $e_1\cdot\pd_1R_\psi\tp{w,h}\le 0$. Since $e_1\cdot\pd_1R_0\tp{w,h} = 1$, the portion of Theorem~\ref{1_MAIN_THM}'s solution set $\mathcal{C}$ within a neighborhood of $\psi = 0$ are necesarily graphical, but this may not remain true for the larger waves. While in this paper we do not explicitly prove the existence or lack thereof of traveling wave solutions to~\eqref{traveling_FBIDF} that are not graphical, our conformal formulation and construction opens the door to future study on the possible development of overturning traveling waves along such solution sets.

Let us now enter a brief overarching discussion on the proofs of our main results. The strategy is to employ a global implicit function theorem (GIFT) based on the Leray-Schauder degree to the operator determined by the left hand side of the equation~\eqref{reformulation for the coordinate function psi}. Similar techniques have been used in recent years for the construction of large solutions to other physical nonlinear PDEs, e.g.~\cite{MR3949722,MR4135095,MR4557691,NGCD}. However, the particular GIFTs implemented in these results, that are a special case of Theorem II.6.1 in Kielh\"{o}fer~\cite{MR2859263}, require the operators to manifest as nonlinear compact perturbations of the identity. Not all versions of GIFTs in the literature possess the aforementioned structural mandate; in particular, the class of analytic global implicit function and bifurcation theorems (see, e.g. \cite{MR1956130,MR3765551,CWW} and references therein) are more flexible in the forms of the nonlinear operators that they can handle; but, as the name suggests, real analyticity of the operator is required. We emphasize that analytic GIFT methods were employed in the traveling surface wave problem for the viscous shallow water equations in~\cite{SWB}.

However, the particular nonlinear equation~\eqref{reformulation for the coordinate function psi} is simply not within the scope of analytic GIFTs for the reason that the nonlinear operator fails to be analytic under natural smoothness assumptions on the data profiles $\mathsf{f}$ and $\upphi$. The reason is that the numerous composition-type nonlinearities control the maximal order of Fr\'echet differentiability of the nonlinear operator in terms of the data profile regularity. We would have to then select real analytic data to get a real analytic operator, which is, in the authors' opinion, an unnecessarily restrictive choice. Thus, wanting to use data of a reasonable finite amount of regularity, we are led back to the convenient choice of using a real-variable GIFT based on the Leray-Schauder degree. Working towards massaging the equations in the form of a compact perturbation of the identity, we first identify a `quasilinearization' of the operator of~\eqref{reformulation for the coordinate function psi}. Letting $\kappa = \upkappa^2$ and writing this equation abstractly as $\bf{P}\tp{\upkappa,\psi} = 0$, we then are able to decompose the operator $\bf{P}$ as
\begin{equation}\label{quasilinearization of the equation}
    \bf{P}\tp{\upkappa,\psi} = \bf{A}\tp{\upkappa,\psi}\psi + \bf{Q}\tp{\upkappa,\psi}
\end{equation}
with $\bf{A}$, defined in~\eqref{the_main_part_A}, valued in the space of linear operators $\z{C}^{1,\al}\tp{\T}\to\z{C}^\al\tp{\T}$ and $\bf{Q}$, defined in~\eqref{the remainder operator Q}, a nonlinear and compact map $\z{C}^{1,\al}\tp{\T}\to\z{C}^\al\tp{\T}$. 

As it turns out, the operator $\bf{A}\tp{\upkappa,\psi}$ appearing in~\eqref{quasilinearization of the equation} is invertible with suitable estimates if we have satisfied the hidden ellipticity conditions~\eqref{conditions for the hidden ellipticity}. In these cases, we can apply the inverse operator $\tsb{\bf{A}\tp{\upkappa,\psi}}^{-1}$ to the equation $\bf{P}\tp{\upkappa,\psi} = 0$ to derive the equivalent functional form
\begin{equation}
\bf{F}\tp{\upkappa,\psi} = \psi + \tsb{\bf{A}\tp{\upkappa,\psi}}^{-1}\bf{Q}\tp{\upkappa,\psi}=0.
\end{equation}
We are then able to prove that the difference between $\bf{F}$ and the identity is indeed a compact nonlinear operator. We also calculate that the Fr\'echet derivative $D_2\bf{F}\tp{0,0}:\z{C}^{1,\al}\tp{\T}\to\z{C}^{1,\al}\tp{\T}$ is an isomorphism. Therefore, we can apply the global implicit function theorem recorded in Theorem 4.6 of~\cite{NGCD}. Upon doing so we are granted a connected solution set
\begin{equation}
\mathcal{C}\subset\tcb{\bf{F} = 0}\subset\R\times\z{C}^{1,\al}\tp{\T}
\end{equation}
such that $\mathcal{C}$ is either unbounded or $\mathcal{C}\setminus\tcb{\tp{0,0}}$ is connected. This latter alternative is what is known as the closed loop. We show that the only solution to $\bf{F}\tp{0,\psi} = 0$ is $\psi = 0$ and deduce that the closed loop alternative does not occur. So indeed the former alternative holds and the function
\begin{equation}\label{initial blow up quantity right here}
    \mathcal{C}\ni\tp{\upkappa,\psi}\mapsto|\upkappa| + \tnorm{\psi}_{C^{1,\al}\tp{\T}}\in\tp{0,\infty}
\end{equation}
is unbounded. Let us pause to make a comparison with the construction of large traveling waves with surface tension in \cite{NGCD}, that used the same GIFT. In the presence of surface tension, by inverting the capillary-gravity operator $H+\m{g}I$, $H$ being the mean curvature operator, the needed compactness is immediate due to the two derivative gain from the inverse $(H+\m{g}I)^{-1}$. In the present work, as hinted above, the compactness is delicate and comes from the hidden ellipticity, which leads to the invertibility of $\bf{A}$, and the commutator structure of $\bf{Q}$.

The next step is to show that the initial blow-up quantity~\eqref{initial blow up quantity right here} can be refined by replacing the high norm $\tnorm{\cdot}_{C^{1,\al}\tp{\T}}$ appearing above with a weaker norm $\tnorm{\cdot}_X$ such that $C^{1,\al}\tp{\T}\emb X$. The goal is to take $X$ as large as possible, as to learn as much as we can about the nature of the blow-up. We achieve this through careful paradifferential commutator estimates and are able to take $X = H^{1/2 + \ep}\tp{\T}$ for any $0<\ep\le 1/2$. Under the additional structural assumption that the bulk force $\mathsf{f}$ vanishes identically, we can sharpen even further via similar methods and take $X = C^\ep\tp{\T}$ for any $0<\ep\le1/2$.

Finally, we are only interested in the members of $\mathcal{C}$ that correspond to meaningful solutions to the original free boundary problem, and so we need to study $\mathcal{C}\cap\tp{\R\times\mathcal{U}}$, where $\mathcal{U}$ is the set of admissible domain coordinate functions~\eqref{the set of admissible domain coordinate functions}. In other words, we need to understand what is happening with the Riemann mappings $R_\psi$ at the possible scenario when large $\psi$ in $\mathcal{C}$ are leaving the set $\mathcal{U}$; when this is happening, the Riemann mappings are failing extend to the boundary. Using tools from classical complex analysis, we are able to deduce that $R_\psi$ can \emph{only} fail to be a diffeomorphism up to the boundary if the distortion $\mathscr{D}\tp{\psi}$ (defined in~\eqref{the distortion}) is becoming unboundedly large. Unpacking the meaning of $\mathscr{D}\tp{\psi}\to\infty$, shows that either the top free boundary of $\Omega_\psi$ becomes arbitrarily close to self-intersection or the conformal parametrizations of the free surfaces~\eqref{parametrization of the top boundary component of the image} degenerate via a vanishing derivative.

The remainder of the paper, following a brief digression on notation in Section~\ref{subsection on notational conventions}, is organized into three body sections. In Section~\ref{subsection on preliminaries} we record a flurry of technical estimates and mapping properties of the constituent pieces of the operator on the left hand side of equation~\eqref{reformulation for the coordinate function psi}. Having developed an understating of the nuances of the objects involved, we then proceed with an unencumbered construction and analysis of large solutions to the operator equation in Section~\ref{section on analysis of the nonlinear operator}. Finally, in Section~\ref{section on refinements}, we explore conformal degeneracy and complete the proofs of Theorems~\ref{1_MAIN_THM} and~\ref{2_MAIN_THM}.

\subsection{Conventions of notation}\label{subsection on notational conventions}

The natural numbers are $\N=\tcb{0,1,2,3,\dots}$ and we let $\N^+ = \N\setminus\tcb{0}$. The positive real numbers are $\R^+ = \tp{0,\infty}$. For sets $E_0\subset E_1$ we let $\mathds{1}_{E_0}:E_1\to\tcb{0,1}$ denote the characteristic function of $E_0$. If $X(\T)$ is a space of functions on the $1$-torus $\T$, we let $\z{X}\tp{\T}$ denote the subspace of function with vanishing integral over $\T$. We let $C^{k,\al}$, for $k\in\N$ and $0<\al<1$, denote the usual Banach space of $k$-times differentiable functions with $\al$-H\"{o}lder continuous derivatives. When $k = 0$, we shall simply abbreviate $C^{0,\al} = C^\al$.

We shall also use the Besov spaces of periodic functions. The most convenient way to describe these spaces is through Littlewood-Paley decompositions. Upon letting $\tcb{\Updelta_j}_{j\in\N}$ denote a standard sequence of inhomogeneous Littlewood-Paley dyadic frequency localization operators, e.g. as in equation~\eqref{LP localization operators} below, the norm on the Besov space $B^{s,p_1}_{p_2}\tp{\T}$ is
\begin{equation}\label{I_LOVE_BESOV_SPACES}
    \tnorm{f}_{B^{s,p_1}_{p_2}\tp{\T}} = \tnorm{\tcb{2^{js}\tnorm{\Updelta_j f}_{L^{p_1}\tp{\T}}}_{j\in\N}}_{\ell^{p_2}\tp{\Z}}\text{ for }s\in\R,\;p_1,p_2\in[1,\infty].
\end{equation}
When $p_1 = p_2 = p<\infty$ and $s\in\R\setminus\N$ we write instead $B^{s,p_1}_{p_2}\tp{\T} = W^{s,p}\tp{\T}$. For more information on Besov spaces, we refer the reader to~\cite{MR730762,MR2768550,MR3726909,MR4567945} and references therein. We mention that when $s = k+\al$ for $k\in\N$ and $0<\al<1$ we have the equality $C^{k,\al}\tp{\T} = B^{s,\infty}_\infty\tp{\T}$ with equivalence of norms.

\section{Preliminaries}\label{subsection on preliminaries}

Our construction of large traveling wave solutions to the equations of free boundary incompressible Darcy flow requires a prelude of estimates and property developments within a suitable functional framework. It is throughout this section that these technical preliminaries are recorded.

\subsection{Technical estimates and mapping properties}\label{subsection on technical estimates and mapping properties}

In this subsection we consider the atomic constituents of the operator appearing on the left hand side of equation~\eqref{reformulation for the coordinate function psi} and develop their necessary estimates and mapping properties. We begin by enumerating the mapping properties of the Cauchy-Riemann solver $\mathcal{E}$ that was introduced in equation~\eqref{definition of the Cauchy Riemann solver}.
\begin{lem}[Boundedness of the Cauchy-Riemann solver]\label{lemma on bounds on the Cauchy Riemann solver}
    For any $\R\ni s\ge1/2$, $\upiota\in\tcb{0,1}$, and $0<\al<1$ the linear maps
    \begin{equation}\label{CR mapping estimates}
        \z{H}^s\tp{\T}\ni\psi\mapsto\mathcal{E}\psi\in H^{1/2+s}\tp{\T\times\tp{0,h};\R^2}\text{ and }\z{C}^{\upiota,\al}\tp{\T}\ni\psi\mapsto\mathcal{E}\psi\in C^{\upiota,\al}\tp{\T\times\tp{0,h};\R^2}
    \end{equation}
    are bounded. Moreover, for all $\psi$ belonging to one of the domains of the above maps we have that $\mathcal{E}\psi$ is smooth in the interior and solves the Cauchy-Riemann equations
    \begin{equation}\label{CR bulk}
        e_1\cdot\pd_1\mathcal{E}\psi = e_2\cdot\pd_2\mathcal{E}\psi\text{ and }-e_1\cdot\pd_2\mathcal{E}\psi = e_2\cdot\pd_1\mathcal{E}\psi\text{ in }\T\times\tp{0,h}
    \end{equation}
    along with the boundary conditions
    \begin{equation}\label{CR boundary}
        e_2\cdot\m{Tr}_{\T\times\tcb{h}}\mathcal{E}\psi = \psi,\;e_2\cdot\m{Tr}_{\T\times\tcb{0}}\mathcal{E}\psi = 0,\;\int_{\T}e_1\cdot\m{Tr}_{\T\times\tcb{h}}\mathcal{E}\psi = 0.
    \end{equation}
\end{lem}
\begin{proof}
    The identities~\eqref{CR bulk} and~\eqref{CR boundary} are simple and direct calculations.
    
    The left hand mapping property claimed in~\eqref{CR mapping estimates} is true for $s\in1/2 + \N$ via a direct calculation based on Plancherel's theorem. The intermediate cases then follow by interpolation.

    We turn our attention to proving the right hand mapping properties asserted in equation~\eqref{CR mapping estimates}. We shall factor the operator $\mathcal{E}$ into simpler components. To wit, let us define, for $z\in[0,h]$, the auxiliary mappings
    \begin{equation}\label{bounds of the more aux ops}
        T:C^{\upiota,\al}\tp{\T}\times C^{\upiota,\al}\tp{\T}\to C^{\upiota,\al}\tp{\T\times\tp{0,h}}\text{ and }H_z:C^{\upiota,\al}\tp{\T}\to C^{\upiota,\al}\tp{\T}
    \end{equation}
    with actions defined for $g^1,g^2,g^3\in C^{\upiota,\al}\tp{\T}$ via
    \begin{equation}\label{more aux ops}
        T(g^1,g^2) = u\text{ and }\mathscr{F}[H_zg^3]\tp{\xi} = -\ii\mathds{1}_{\Z\setminus\tcb{0}}\tp{\xi}\f{\xi\cosh\tp{2\pi|\xi|z}}{\tabs{\xi}\sinh\tp{2\pi|\xi|h}}\mathscr{F}[g^3]\tp{\xi}
    \end{equation}
    where $u$ is the unique solution to the boundary value problem
    \begin{equation}
        \Delta u = 0\text{ in }\T\times\tp{0,h},\;u = g^1\text{ on }\T\times\tcb{0},\text{ and }u = g^2\text{ on }\T\times\tcb{h}.
    \end{equation}
    The boundedness of $H_z$ between the spaces stated in~\eqref{bounds of the more aux ops} is clear thanks to well-known mapping properties of Fourier multipliers acting between H\"{o}lder spaces; see, for instance, the analysis of Section 5.3 in Stein~\cite{MR1232192} or Appendix A in~\cite{NGCD}. On the other hand, that $T$ is well-defined and bounded as a linear mapping acting between the stated spaces of~\eqref{bounds of the more aux ops} follows from classical Schauder theory (see, e.g, Gilbarg and Trudinger~\cite{MR1814364} or Simon~\cite{MR1459795}).

    In terms of the mappings $T$ and $H_z$ of~\eqref{more aux ops}, the Cauchy-Riemann solver $\mathcal{E}$ decomposes via
    \begin{equation}\label{decomposition of the Cauchy Riemann operator}
        \mathcal{E}\psi = e_1T(H_0\psi,H_h\psi) + e_2T(0,\psi)
    \end{equation}
    for appropriate $\psi$. The boundedness stated in the right hand side of~\eqref{CR mapping estimates} follows from~\eqref{bounds of the more aux ops} and~\eqref{decomposition of the Cauchy Riemann operator}.
\end{proof}

Our next task is to understand the Nemytskii-type operators $\psi\mapsto\mathsf{f}\circ R_\psi$ and $\psi\mapsto\upphi\circ R_\psi$ appearing in equation~\eqref{reformulation for the coordinate function psi}, where $R_\psi$ is defined in~\eqref{The Coordinate Maps}; we shall prove these maps to be continuously differentiable between suitable Banach spaces of functions if $\mathsf{f}$ and $\mathsf{\upphi}$ are taken fixed and sufficiently regular.

\begin{lem}[Mapping properties of the composition operators]\label{lem on mapping properties of the composition operators}
    Let $\mathsf{f}\in C^4\tp{\T\times\R;\R^2}$, $\upphi\in C^4\tp{\T\times\R}$, $\upiota\in\tcb{0,1}$, $0<\al,\ep<1$. The following mappings are well-defined, continuously differentiable in the Fr\'echet sense, and map bounded sets to bounded sets:
    \begin{multline}\label{map_1}
        \z{C}^{\upiota,\al}\tp{\T}\ni\psi\mapsto\\
        \tp{\tp{\grad\cdot\mathsf{f}}\circ R_\psi,\m{Tr}_{\T\times\tcb{0}}\mathsf{f}\circ R_\psi,\m{Tr}_{\T\times\tcb{h}}\mathsf{f}\circ R_\psi}\in C^{\upiota,\al}\tp{\T\times\tp{0,h}}\times C^{3}\tp{\T;\R^2}\times C^{\upiota,\al}\tp{\T;\R^2},
    \end{multline}
    \begin{equation}\label{map_2}
        \z{C}^{\upiota,\al}\tp{\T}\ni\psi\mapsto\m{Tr}_{\T\times\tcb{h}}\upphi\circ R_{\psi}\in C^{\upiota,\al}\tp{\T},
    \end{equation}
    \begin{multline}\label{map_3}
        \z{H}^{1/2 + \ep}\tp{\T}\ni\psi\mapsto
        \tp{\tp{\grad\cdot \mathsf{f}}\circ R_\psi,\m{Tr}_{\T\times\tcb{0}}\mathsf{f}\circ R_\psi,\m{Tr}_{\T\times\tcb{h}}\mathsf{f}\circ R_\psi}\\\in C^{\ep}\tp{\T\times\tp{0,h}}\times C^{3}\tp{\T;\R^2}\times H^{1/2 + \ep}\tp{\T;\R^2},
    \end{multline}
    and
    \begin{equation}\label{map_4}
        \z{H}^{1/2 + \ep}\tp{\T}\ni \psi\mapsto\m{Tr}_{\T\times\tcb{h}}\upphi\circ R_{\psi}\in H^{1/2 + \ep}\tp{\T}.
    \end{equation}
\end{lem}
\begin{proof}
    The compositional results that we seek are essentially well-known in the literature, but are difficult to find stated in the precise form we require.

    We begin by defining an auxiliary affine linear operator. Let $\mathfrak{E}$ denote the extension operator constructed by Theorem 5 of Chapter 6 in Stein~\cite{MR290095}. We have that
    \begin{equation}
        \mathfrak{E}:C^{\upiota,\al}\tp{\R\times(0,h)}\to C^{\upiota,\al}\tp{\R^2}
    \end{equation}
    maps boundedly. Let also $\chi\in C^\infty\tp{\R}$ satisfy $\chi(w) = 1$ for $|w|\le 1$ and $\chi(w) = 0$ for $|w|\ge 2$, $w\in\R$. Now define
    \begin{equation}\label{yet_another_helper}
        \mathcal{A}:\z{C}^{\upiota,\al}\tp{\T}\to C^{\upiota,\al}\tp{\R^2;\R^2}\text{ via }\mathcal{A}\tp{\psi} = \mathfrak{E}\tp{\chi R_\psi}\text{ for all }\psi\in \z{C}^{\upiota,\al}\tp{\T}
    \end{equation}
    with the understanding that $R_\psi\in C^{\upiota,\al}\tp{\R\times\tp{0,h};\R^2}$ via the natural periodic lifting and $\tp{\chi R_\psi}\tp{w,z} = \chi(w)R_\psi(w,z)$ for $(w,z)\in\R\times\tp{0,h}$.

    By using the map $\mathcal{A}$, a restriction operator, and Lemma~\ref{lemma on bounds on the Cauchy Riemann solver}, we deduce the following reduction for any fixed $\mathsf{J}\in C^3\tp{\T\times\R}$: The operator 
    \begin{equation}\label{thing_we_want}
        \z{C}^{\upiota,\al}\tp{\T}\ni\psi\mapsto\mathsf{J}\circ R_\psi\in C^{\upiota,\al}\tp{\T\times\tp{0,h}}
    \end{equation}
    is well-defined, maps bounded sets to bounded sets, and is $C^1$ in the Fr\'echet sense if the (closely related) operator
    \begin{equation}\label{thing_we_can_do}
        C^{\upiota,\al}\tp{\R^2}\ni\eta\mapsto\mathsf{J}\circ\eta\in C^{\upiota,\al}\tp{\R^2}.
    \end{equation}
    is well-defined, maps bounded sets to bounded sets, and is $C^1$ in the Fr\'echet sense. Indeed, assuming we know~\eqref{thing_we_can_do}, we take $\eta = \mathcal{A}\tp{\psi}$ and consider the restriction of the composition $\mathsf{J}\circ\eta$ to deduce~\eqref{thing_we_want}.

    The claimed mapping properties of the `lifted' composition operator~\eqref{thing_we_can_do} are well-known; see, for instance, Sections 3 and 4 in Lanza de Cristoforis~\cite{MR1307964}. See also Dr\'{a}bek~\cite{MR380547} and Bourdaud and Lanza de Cristoforis~\cite{MR1926867}.

    We take $\mathsf{J} = e_i\cdot\mathsf{f}$ ($i\in\tcb{1,2}$) or $\mathsf{J} = \grad\cdot\mathsf{f}$ in~\eqref{thing_we_want} to deduce that 
    \begin{equation}
        \z{C}^{\upiota,\al}\tp{\T}\ni\psi\mapsto\tp{\tp{\grad\cdot\mathsf{f}}\circ R_\psi,\m{Tr}_{\T\times\tcb{h}}\mathsf{f}\circ R_\psi}\in C^{\upiota,\al}\tp{\T\times\tp{0,h}}\times C^{\upiota,\al}\tp{\T;\R^2}
    \end{equation}
    is well-defined, maps bounded sets to bounded sets, and is class $C^1$. This is almost the first claim~\eqref{map_1}. To prove that 
    \begin{equation}\label{thing_we_want_2}
        \z{C}^{\upiota,\al}\tp{\T}\ni\psi\mapsto\m{Tr}_{\T\times\tcb{0}}\mathsf{f}\circ R_{\psi}\in C^3\tp{\T;\R^2}
    \end{equation}
    satisfies the same desired properties, we can argue similarly as above, but now we note that
    \begin{equation}\label{thing_we_observe}
        \tp{\m{Tr}_{\T\times\tcb{0}}\mathsf{f}\circ R_{\psi}}\tp{w} = \mathsf{f}\tp{w + \tp{H_0\psi}\tp{w},0}\text{ for }w\in\T
    \end{equation}
    with $H_0$ as defined in~\eqref{more aux ops}. The symbol of the Fourier multiplier associated with $H_0$ is decaying exponentially quickly in frequency and so $H_0\psi$ is smooth in a quantitative way. Thus the map of~\eqref{thing_we_want_2} is well-defined, class $C^1$, and maps bounded sets to bounded sets as a consequence of the so called `Omega Lemma': see, e.g. 2.4.18 in Abraham, Marsden, and Ratiu~\cite{MR960687}. The claims for the mappings of equation~\eqref{map_1} are now established.

    We now turn our attention to the verification of some of the mapping properties asserted in equation~\eqref{map_3}. By arguing exactly as in~\eqref{thing_we_want_2} and~\eqref{thing_we_observe} we conclude that the desired properties hold for the middle component of~\eqref{map_3}:
    \begin{equation}
        \z{H}^{1/2 + \ep}\tp{\T}\ni\psi\mapsto\m{Tr}_{\T\times\tcb{0}}\mathsf{f}\circ R_{\psi}\in C^3\tp{\T\times\R^2}.
    \end{equation}

    For the first component of~\eqref{map_3}, we use Lemma~\ref{CR mapping estimates} and the Sobolev embedding
    \begin{equation}
        H^{1 + \ep}\tp{\T\times\tp{0,h}}\emb C^\ep\tp{\T\times\tp{0,h}}
    \end{equation}
    to conclude that $\mathcal{A}$ of~\eqref{yet_another_helper} restricts to a bounded affine linear operator
    \begin{equation}\label{I_am_bored}
        \mathcal{A}:\z{H}^{1/2 + \ep}\tp{\T}\to C^\ep\tp{\R^2;\R^2}.
    \end{equation}
    Thus, we may combine~\eqref{I_am_bored} with the reduction described in equations~\eqref{thing_we_want} and~\eqref{thing_we_can_do} once more to conclude that 
    \begin{equation}
        \z{H}^{1/2 + \ep}\tp{\T}\ni\psi\mapsto\tp{\grad\cdot\mathsf{f}}\circ R_{\psi}\in C^\ep\tp{\T\times\tp{0,h}}
    \end{equation}
    is well-defined, Fr\'echet continuously differentiable, and maps bounded sets to bounded sets.
    
    It remains to analyze~\eqref{map_2}, \eqref{map_4}, and the final component of~\eqref{map_3}. That the desired properties for the map~\eqref{map_2} hold, is a consequence of a similar argument to the one above that relies on the results of~\cite{MR1307964}; thus, we omit the repetitive details. The sought after properties for~\eqref{map_4} and the final component of~\eqref{map_3}, after an extension, cut-off, and restriction argument, follow from the fact that the composition operator
    \begin{equation}
        H^{1/2 + \ep}\tp{\R}\ni\eta\mapsto\mathsf{J}\circ\eta\in H^{1/2 + \ep}\tp{\R}
    \end{equation}
    is well-defined, class $C^1$, and maps bounded sets to bounded sets. This latter fact is a consequence of the estimates of Section 4 of Chapter 2 in Taylor~\cite{MR1766415} and the converse to Taylor's theorem, see, e.g., Section 2.4B in Abraham, Marsden, and Ratiu~\cite{MR960687}; see also Bourdaud and Lanza de Cristoforis~\cite{MR2369144}.
\end{proof}

We continue our exploration of the preliminary material by next examining the mapping properties of the final term in~\eqref{reformulation for the coordinate function psi} involving the operator $S$ that we recall is defined in equation~\eqref{help_solve_2}. The following result, in particular, verifies that this term is lower order in the nonlinear pseudodifferential equation~\eqref{reformulation for the coordinate function psi}.

\begin{lem}[Analysis of the lower-order bulk remainder]\label{lem on analysis of the lower order bulk remainder}
    Let $\mathsf{f}\in C^4\tp{\T\times\R;\R^2}$, $0<\al<1$, and $0<\ep\le1/2$. The mapping
    \begin{equation}\label{defn_l_o_t}
        \psi\mapsto\bf{K}\tp{\psi} = S\tp{\tp{\tabs{\grad R_\psi}^2/2}\tp{\grad\cdot\mathsf{f}}\circ R_\psi,\m{Tr}_{\T\times\tcb{0}}\tp{\mathsf{f}\circ R_\psi}\cdot\tp{\pd_1R_\psi}^\perp}
    \end{equation}
    is well-defined, continuously differentiable in the Fr\'echet sense, and maps bounded sets to bounded sets for the following domain and codomain pairs:
    \begin{equation}\label{map_l_o_t}
        \bf{K}:\z{C}^{1,\al}\tp{\T}\to C^{1,\al}\tp{\T},\;\bf{K}:\z{H}^{1/2 + \ep}\tp{\T}\to W^{\ep,1/\tp{1 - \ep}}\tp{\T},\;\text{and }\bf{K}:\z{H}^{1 + \al/2}\tp{\T}\to C^\al\tp{\T}.
    \end{equation}
\end{lem}
\begin{proof}
    Let us begin with the proof of the left hand mapping property claimed in equation~\eqref{map_l_o_t}. We first note that by standard Schauder theory (see, e.g., \cite{MR1814364,MR1459795}), the linear map $S$ (defined in~\eqref{help_solve_2}), maps boundedly
    \begin{equation}\label{S_SCHAUDER}
        S:C^{\al}\tp{\T\times\tp{0,h}}\times C^{1,\al}\tp{\T}\to C^{1,\al}\tp{\T}.
    \end{equation}
    and, therefore, it is sufficient to observe that
    \begin{equation}\label{UGH_THIS_SECTION}
        \z{C}^{1,\al}\tp{\T}\ni\psi\mapsto\tp{\tp{\tabs{\grad R_\psi}^2/2}\tp{\grad\cdot\mathsf{f}}\circ R_
        \psi,\m{Tr}_{\T\times\tcb{0}}\tp{\mathsf{f}\circ R_\psi}\cdot\tp{\pd_1R_\psi}^\perp}\in C^{\al}\tp{\T\times\tp{0,h}}\times C^{1,\al}\tp{\T}
    \end{equation}
    is a well-defined mapping that is class $C^1$ and maps bounded sets to bounded sets. This latter fact, in turn, is a consequence of Lemmas~\ref{lemma on bounds on the Cauchy Riemann solver} and~\ref{lem on mapping properties of the composition operators}, simple algebra properties of the spaces involved, and the observation that 
    \begin{equation}
        \z{C}^{1,\al}\tp{\T}\ni\psi\mapsto\m{Tr}_{\T\times\tcb{0}}\tp{\pd_1R_\psi}^\perp\in C^{1,\al}\tp{\T;\R^2}
    \end{equation}
    is, in fact, an affine linear smoothing operator.

    We now focus on the claimed middle hand mapping property of equation~\eqref{map_l_o_t}. A key point is that the map $S$, now thought of as the function
    \begin{equation}
        S:L^{1/\tp{1-\ep}}\tp{\T\times\tp{0,h}}\times W^{\ep,1/\tp{1-\ep}}\tp{\T}\to W^{\ep,1/\tp{1 - \ep}}\tp{\T}
    \end{equation}
    remains well-defined and bounded, thanks to elliptic estimates in Lebesgue spaces and Sobolev space trace theory (see, e.g., \cite{MR1814364,MR3726909}). Since $\psi\mapsto\m{Tr}_{\T\times\tcb{0}}R_\psi$ is a smoothing operator, the second component of the map~\eqref{UGH_THIS_SECTION} again gives us no trouble in verifying that the function
    \begin{equation}
        \z{H}^{1/2 + \ep}\tp{\T}\ni\psi\mapsto\m{Tr}_{\T\times\tcb{0}}\tp{\mathsf{f}\circ R_\psi}\cdot\tp{\pd_1R_\psi}^\perp\in W^{\ep,1/\tp{1-\ep}}\tp{\T}
    \end{equation}
    has the desired mapping properties.

    By Lemmas~\ref{lemma on bounds on the Cauchy Riemann solver} and~\ref{lem on mapping properties of the composition operators}, the Sobolev embedding
    \begin{equation}
        H^\ep\tp{\T\times\tp{0,h}}\emb L^{2/(1-\ep)}\tp{\T\times\tp{0,h}},
    \end{equation}
    and the boundedness of the product map
    \begin{equation}
        L^{2/\tp{1 - \ep}}\tp{\T\times\tp{0,h}}\times L^{2/\tp{1 - \ep}}\tp{\T\times\tp{0,h}}\to L^{1/\tp{1 - \ep}}\tp{\T\times\tp{0,h}}
    \end{equation}
    we deduce that
    \begin{equation}
        \z{H}^{1/2 + \ep}\tp{\T}\ni\psi\mapsto\tp{\tabs{\grad R_\psi}^2/2}\tp{\grad\cdot\mathsf{f}}\circ R_\psi\in L^{1/\tp{1-\ep}}\tp{\T\times\tp{0,h}}
    \end{equation}
    is a well-defined mapping that is Fr\'{e}chet $C^1$ and maps bounded sets to bounded sets. Upon synthesizing the above, we complete the proof of the middle mapping property of equation~\eqref{map_l_o_t}.

    The proof of the final mapping property of~\eqref{map_l_o_t} is much the same. By similar arguments, we deduce that the following mappings are well-defined and map bounded sets to bounded sets:
    \begin{equation}\label{RR_0}
        \z{H}^{1 + \al/2}\tp{\T}\ni\psi\mapsto\tp{\tp{\tabs{\grad R_\psi}^2/2}\tp{\grad\cdot\mathsf{f}}\circ R_\psi,\m{Tr}_{\T\times\tcb{0}}\tp{\mathsf{f}\circ R_\psi}\cdot\tp{\pd_1 R_\psi}^\perp}\in L^{2/(1-\al)}\tp{\T}\times W^{(1+\al)/2,2/(1-\al)}\tp{\T}
    \end{equation}
    and
    \begin{equation}\label{RR_1}
        S:L^{2/\tp{1 - \al}}\tp{\T}\times W^{(1+\al)/2,2/(1-\al)}\tp{\T}\to W^{(1+\al)/2,2/(1-\al)}\tp{\T}.
    \end{equation}
    So the result follows by composing~\eqref{RR_0} and~\eqref{RR_1} and appealing to the embedding
    \begin{equation}
        W^{(1+\al)/2,2/(1-\al)}\tp{\T}\emb C^\al\tp{\T}.
    \end{equation}
\end{proof}

The final result of this subsection develops important commutator estimates for certain simple Fourier multipliers that we introduce now. For suitable functions $\phi$ we define $\mathbb{P}_\pm\phi$ via
\begin{equation}\label{definition of P_plus_minus}
    \mathscr{F}[\mathbb{P}_\pm\phi]\tp{\xi} = \mathds{1}_{\N^+}\tp{\pm\xi}\mathscr{F}\tsb{\phi}\tp{\xi},\quad\xi\in\Z.
\end{equation}
Note that $\P_+\phi + \P_-\phi = \phi - \int_{\T}\phi$; moreover if $\phi$ is $\R$ valued, then $\P_+\phi$ is $\R$-valued if and only if it vanishes identically.
\begin{prop}[Commutator estimates for $\mathbb{P}_\pm$]\label{prop on Pplus commutators}
    For $0<\ep\le1/2$, $\ep<\del$, there exists $C\in\R^+$ such that
\begin{equation}\label{comu_1}
  \tnorm{\P_\pm\tp{f\pd_1g} - f\P_\pm\pd_1g}_{W^{\ep,1/\tp{1 - \ep}}\tp{\T}}\le C\tnorm{f}_{H^{1/2 + \del}\tp{\T}}\tnorm{g}_{H^{1/2 + \ep}\tp{\T}}.
    \end{equation}
For $\beta, \gamma \in \R$, there exists $C\in \R^+$ such that     \begin{equation}\label{comu_3}
       \begin{cases} \tnorm{\P_\pm\tp{f\pd_1g} - f\P_\pm\pd_1g}_{B^{\beta+\gamma-1, \infty}_\infty\tp{\T}}\le C\|f\|_{B^{\gamma, \infty}_\infty(\T)}\tnorm{g}_{C^\be\tp{\T}}\quad\text{if~}\beta<1~\text{and~}\beta+\gamma>1,\\
        \tnorm{\P_\pm\tp{f\pd_1g} - f\P_\pm\pd_1g}_{B^{\gamma, \infty}_\infty\tp{\T}}\le C\|f\|_{B^{\gamma, \infty}_\infty(\T)}\tnorm{g}_{C^1\tp{\T}}\quad\text{if~}\gamma>0.
        \end{cases}
    \end{equation}
\end{prop}
\begin{proof}
    Before we enter the heart of the proof, we must first introduce the necessary notation associated with standard Littlewood-Paley decompositions. Let us fix a pair of smooth functions $\chi,\tilde{\chi}\in C^\infty\tp{\R}$ such that: $\tilde{\chi}\tp{\xi} = 0$ when $|\xi|\ge6/5$, $\chi(\xi) = 1$ for $1\le|\xi|\le 2$, $\chi\tp{\xi} = 0$ for $|\xi|>11/5$ and $|\xi|<4/5$, and
    \begin{equation}\label{the conditions on the chi_bros}
        \tilde{\chi}\tp{\xi} + \sum_{j=0}^\infty\chi(\xi/2^j) = 1\text{ for all }\xi\in\R.
    \end{equation}
    The associated sequences of Fourier multiplication operators $\tcb{\Updelta_j}_{j\in\N}$ and $\tcb{\Upsigma_j}_{j\in\N}$ are then defined via
    \begin{equation}\label{LP localization operators}
        \mathscr{F}[\Updelta_j\phi]\tp{\xi} = \begin{cases}
            \tilde{\chi}\tp{\xi}\mathscr{F}[\psi]\tp{\xi}&\text{if }j=0,\\
            \chi\tp{\xi/2^{j-1}}\mathscr{F}[\psi]\tp{\xi}&\text{if }j>0,
        \end{cases}\text{ and }\Upsigma_j = \sum_{i=0}^j\Updelta_i.
    \end{equation}
    The identity~\eqref{the conditions on the chi_bros} ensures that the Littlewood-Paley decomposition $\phi = \sum_{j=0}^\infty\Updelta_j\phi$ holds in a suitable sense for sufficiently regular $\phi$.

    We only prove estimates~\eqref{comu_1} and~\eqref{comu_3} for the case $\P_+$; the analysis for $\P_-$ is identical. Let us now write
    \begin{equation}\label{short_hand_bilinear}
        \Lambda^+\tp{f,g} = \P_+\tp{f\pd_1g} - f\P_+\pd_1g
    \end{equation}
    to denote the action of the bilinear operator under consideration. By Fourier support considerations, we deduce that~\eqref{short_hand_bilinear} is truly a high-low paraproduct in the sense that 
    \begin{equation}\label{high-low form of the commutator}
        \Lambda^+\tp{f,g} = \sum_{j=0}^\infty\Lambda^+\tp{\Updelta_jf,\Upsigma_{j+10}g}.
    \end{equation}
    The $j^{\m{th}}$ term in the series above is supported in the centered interval of length $2^{j+20}$.

    We first prove the claimed estimate~\eqref{comu_1}. By using~\eqref{high-low form of the commutator}, support considerations, and H\"older's inequality, we first estimate for any $k\in\N$
    \begin{equation}\label{eq_00}
        \tnorm{\Updelta_k\tp{\Lambda^+\tp{f,g}}}_{L^{1/\tp{1 - \ep}}\tp{\T}}\lesssim\sum_{j\ge\max\tcb{0,k-20}}\tnorm{\Updelta_jf}_{L^{2/(1-2\ep)}\tp{\T}}\tnorm{\Upsigma_{j+10}\pd_1g}_{L^2\tp{\T}}.
    \end{equation}
    Then, thanks to Bernstein-type inequalities, we may further bound for each $j$
    \begin{equation}\label{eq_01}
        \tnorm{\Updelta_jf}_{L^{2/(1-2\ep)}\tp{\T}}\lesssim 2^{\ep j}\tnorm{\Updelta_jf}_{L^2\tp{\T}}\text{ and }\tnorm{\Upsigma_{j+10}\pd_1g}_{L^2\tp{\T}}\lesssim 2^{\tp{1/2 - \ep}j}\tnorm{g}_{H^{1/2 + \ep}\tp{\T}}.
    \end{equation}
    By combining~\eqref{eq_00} and~\eqref{eq_01} with the Littlewood-Paley norm on $W^{\ep,1/(1-\ep)}\tp{\T}$ (see~\eqref{I_LOVE_BESOV_SPACES}), we get
    \begin{equation}\label{eq_02}
        \tnorm{\Lambda^+(f,g)}_{W^{\ep,1/(1-\ep)}\tp{\T}}\lesssim\tnorm{g}_{H^{1/2 + \ep}\tp{\T}}\bp{\sum_{k=0}^\infty\bp{2^{\ep k}\sum_{j\ge\max\tcb{0,k-20}}2^{j/2}\tnorm{\Updelta_j f}_{L^2\tp{\T}}}^{1/(1-\ep)}}^{1-\ep}.
    \end{equation}
    In turn, Young's convolution inequality grants us the bound
    \begin{equation}\label{eq_03}
        \bp{\sum_{k=0}^\infty\bp{2^{\ep k}\sum_{j\ge\max\tcb{0,k-20}}2^{j/2}\tnorm{\Updelta_j f}_{L^2\tp{\T}}}^{1/(1-\ep)}}^{1-\ep}\lesssim\tnorm{f}_{B^{1/2 + \ep,2}_{1/(1-\ep)}\tp{\T}}
    \end{equation}
    with the implicit constant depending in particular on $\ep>0$. To conclude~\eqref{comu_1} from~\eqref{eq_02} and~\eqref{eq_03} we use $\del>\ep$ and the embedding
    \begin{equation}
        H^{1/2 + \del}\tp{\T}\emb B^{1/2 + \ep,2}_{1/(1-\ep)}\tp{\T}.
    \end{equation}

    Next, let us justify the estimates in ~\eqref{comu_3}. Similarly to \eqref{eq_00}, we have
   \begin{equation}\label{eq_05}
    \tnorm{\Updelta_k\tp{\Lambda^+\tp{f,g}}}_{L^{\infty}\tp{\T}}\lesssim\sum_{j\ge\max\tcb{0,k-20}}\tnorm{\Updelta_jf}_{L^\infty\tp{\T}}\tnorm{\Upsigma_{j+10}\pd_1g}_{L^\infty\tp{\T}}.
    \end{equation}
    We have
    \begin{equation}
        \tnorm{\Upsigma_{j+10}\pd_1g}_{L^\infty}\lesssim 2^{j(1-\be)}\tnorm{g}_{C^\be\tp{\T}}\text{ and }\tnorm{\Updelta_j f}_{L^\infty\tp{\T}}\lesssim 2^{-j\gam}\tnorm{f}_{B^{\gam,\infty}_{\infty}\tp{\T}}.
    \end{equation}
    It follows that
    \begin{equation}
        \tnorm{\Updelta_k\tp{\Lambda^+\tp{f,g}}}_{L^\infty\tp{\T}}\lesssim 2^{k\tp{1 - \be - \gam}}\tnorm{f}_{B^{\gam,\infty}_{\infty}\tp{\T}}\tnorm{g}_{C^\be\tp{\T}}.
    \end{equation}
    This implies the estimates in~\eqref{comu_3}.
\end{proof}

\subsection{Instantiation of the nonlinear operator and decompositions}\label{subsection on instantiation of the nonlinear operator and decompositions}

This subsection's first result formalizes the left hand side of identity~\eqref{reformulation for the coordinate function psi} as an operator acting between Banach spaces. We recall that the operators $S$ and $G$ are those defined in equation~\eqref{help_solve_2}.

\begin{lem}[Operator encoding]\label{lem on operator encoding}
    Fix $0<\al<1$, $\mathsf{f}\in C^4\tp{\T\times\R;\R^2}$, $\upphi\in C^4\tp{\T\times\R}$, and $\R\ni\m{c},\m{g}\ge0$. The map $\bf{P}:\R\times\z{C}^{1,\al}\tp{\T}\to\z{C}^\al\tp{\T}$ with action given by
    \begin{multline}\label{formal_operator_encoding}
        \bf{P}\tp{\upkappa,\psi} = G\tp{\m{g}\psi + \upkappa^2\m{Tr}_{\T\times\tcb{h}}\upphi\circ R_\psi} - \m{c}\pd_1\psi - \upkappa^2\m{Tr}_{\T\times\tcb{h}}\tp{\mathsf{f}\circ R_\psi}\cdot\tp{\pd_1 R_\psi}^\perp\\
        +\upkappa^2S\tp{\tp{\tabs{\grad R_\psi}^2/2}\tp{\grad\cdot\mathsf{f}}\circ R_\psi,\m{Tr}_{\T\times\tcb{0}}\tp{\mathsf{f}\circ R_\psi}\cdot\tp{\pd_1 R_\psi}^\perp}
    \end{multline}
    is well-defined and continuously differentiable.
\end{lem}
\begin{proof}
    We begin by noting that by standard Schauder theory (see, e.g., \cite{MR1814364,MR1459795}) the linear map $G$, defined in~\eqref{help_solve_2}, maps boundedly $G:C^{1,\al}\tp{\T}\to\z{C}^\al\tp{\T}$. Thus after heeding Lemmas~\ref{lem on mapping properties of the composition operators} and~\ref{lem on analysis of the lower order bulk remainder}, we deduce that $\bf{P}$, when thought of as a mapping into the larger space
    \begin{equation}\label{mapping_into_the_larger_spaces}
        \bf{P}:\R\times\z{C}^{1,\al}\tp{\T}\to C^\al\tp{\T},
    \end{equation}
    is well-defined and continuously differentiable as soon as we verify that
    \begin{equation}\label{I_am_but_a_one_trick_pony}
\z{C}^{1,\al}\tp{\T}\ni\psi\mapsto\m{Tr}_{\T\times\tcb{h}}\tp{\mathsf{f}\circ R_{\psi}}\cdot\tp{\pd_1R_\psi}^\perp\in C^\al\p{\T}
    \end{equation}
    has the same desired mapping properties. But this latter fact is a simple upshot of Lemmas~\ref{lemma on bounds on the Cauchy Riemann solver} and~\ref{lem on mapping properties of the composition operators} and the fact that $C^\al\tp{\T}$ is a Banach algebra.

    To complete the proof we need to check that the map of equation~\eqref{mapping_into_the_larger_spaces} actually has image within the subspace of functions with vanishing zeroth Fourier mode. A straightforward computation using the chain rule and the Cauchy-Riemann equations satisfied by $R_\psi$ shows that
    \begin{equation}\label{the_magic_identity}
        \tp{\tabs{\grad R_\psi}^2/2}\tp{\grad\cdot\mathsf{f}}\circ R_\psi = \grad\cdot\tp{\tsb{\grad R_\psi}^{\m{t}}\tp{f\circ R_\psi}},
    \end{equation}
    with the superscript `$\m{t}$' denoting the matrix transpose. Now let $u\in C^{2,\al}\tp{\T\times\tp{0,h}}$ denote the solution to the left hand equations of~\eqref{help_solve_3} so that
    \begin{equation}\label{not_so_magical}
        \m{Tr}_{\T\times\tcb{h}}\pd_2 u = S\tp{\tp{\tabs{\grad R_\psi}^2/2}\tp{\grad\cdot\mathsf{f}}\circ R_\psi,\m{Tr}_{\T\times\tcb{0}}\tp{\mathsf{f}\circ R_\psi}\cdot\tp{\pd_1 R_\psi}^\perp}.
    \end{equation}
    Integration of the bulk equations satisfied by $u$ over the set $\T\times\tp{0,h}$ when paired with identity~\eqref{the_magic_identity} and the divergence theorem shows
    \begin{multline}\label{intermediate_identity}
       \int_{\T\times\tcb{h}}\pd_2u-\int_{\T\times\tcb{0}} \tp{\mathsf{f}\circ R_\psi}\cdot\tp{\pd_1R_\psi}^\perp = \int_{\T\times\tp{0,h}}\Delta u = \int_{\T\times\tp{0,h}}\f12\tabs{\grad R_\psi}^2\tp{\grad\cdot\mathsf{f}}\circ R_\psi  \\= \int_{\T\times\tcb{h}}\tp{\mathsf{f}\circ R_\psi}\cdot\grad R_\psi e_2-\int_{\T\times\tcb{0}} \tp{\mathsf{f}\circ R_\psi}\cdot\grad R_\psi e_2.
    \end{multline}
    The Cauchy-Riemann equations satisfied by $R_\psi$ also inform us that $\grad R_\psi e_2 = \tp{\pd_1R_\psi}^\perp$; hence, identities~\eqref{intermediate_identity} and~\eqref{not_so_magical} imply that
    \begin{equation}
        \int_{\T}S\tp{\tp{\tabs{\grad R_\psi}^2/2}\tp{\grad\cdot\mathsf{f}}\circ R_\psi,\m{Tr}_{\T\times\tcb{0}}\tp{\mathsf{f}\circ R_\psi}\cdot\tp{\pd_1 R_\psi}^\perp} = \int_{\T\times\tcb{h}}\tp{\mathsf{f}\circ R_\psi}\cdot\tp{\pd_1R_\psi}^\perp.
    \end{equation}
    
    The mapping of $\bf{P}$ into the subspace $\z{C}^\al\tp{\T}$ is now established.
\end{proof}

Note that in the operator encoding~\eqref{formal_operator_encoding} we have replaced the $\kappa$ that appears in~\eqref{reformulation for the coordinate function psi} by $\upkappa^2$. This is done purely for technical reasons related to the desire to have smooth dependence on $\upkappa$ and the nonnegativity $\kappa\ge0$.

Our next result constructs a critically important decomposition for the operator $\bf{P}$ from~\eqref{formal_operator_encoding}. We shall split $\bf{P}$ into the sum of a leading order term and a compact, lower order remainder. Recall the notation for the projection operators $\P_\pm$ introduced in equation~\eqref{definition of P_plus_minus}.

\begin{prop}[Operator decomposition]\label{prop on operator decomposition}
    Let $\al$, $\mathsf{f}$, $\upphi$, $\m{c}$, and $\m{g}$ be as in Lemma~\ref{lem on operator encoding}. For $\upkappa\in\R$ and $\psi\in\z{C}^{\al}\tp{\T}$ we write
    \begin{equation}
        \mathsf{X}_{\upkappa,\psi} = \upkappa^2\m{Tr}_{\T\times\tcb{h}}\tp{\mathsf{f}_1 - \pd_1\upphi}\circ R_\psi - \m{c}\text{ and }\mathsf{Y}_{\upkappa,\psi} = \upkappa^2\m{Tr}_{\T\times\tcb{h}}\tp{\pd_2\upphi - \mathsf{f}_2}\circ R_\psi + \m{g}.
    \end{equation}
    The following hold.
    \begin{enumerate}
        \item For $\upiota\in\tcb{0,1}$ the maps 
        \begin{equation}
            \z{C}^{\upiota,\al}\tp{\T}\ni\psi\mapsto\mathsf{X}_{\upkappa,\psi},\mathsf{Y}_{\upkappa,\psi}\in C^{\upiota,\al}\tp{\T}
        \end{equation}
        are well-defined, continuously differentiable, and map bounded sets to bounded sets.
        \item There exists a continuously differentiable and compact operator $\bf{Q}:\R\times\z{C}^{1,\al}\tp{\T}\to\z{C}^\al\tp{\T}$ such that
        \begin{equation}\label{decomposing_identity_stinky}
            \bf{P}\tp{\upkappa,\psi} = \m{Re}\tsb{\P_+\tp{\tp{\mathsf{X}_{\upkappa,\psi} - \ii\mathsf{Y}_{\upkappa,\psi}}\P_+\pd_1\psi} + \P_-\tp{\tp{\mathsf{X}_{\upkappa,\psi} + \ii\mathsf{Y}_{\upkappa,\psi}}\P_-\pd_1\psi}} + \bf{Q}\tp{\upkappa,\psi}
        \end{equation}
        for all $\upkappa\in\R$ and $\psi\in\z{C}^{1,\al}\tp{\T}$.
        \item Assume that one of the hidden ellipticity conditions of equation~\eqref{conditions for the hidden ellipticity} holds. Then, for all $\upkappa\in\R$ and $\psi\in\z{C}^{\al}\tp{\T}$ the $\C$-valued functions $\mathsf{X}_{\upkappa,\psi} \pm\ii\mathsf{Y}_{\upkappa,\psi}$ are nowhere vanishing and for $\upiota\in\tcb{0,1}$ the maps
        \begin{equation}\label{you_make}
            \R\times\z{C}^{\upiota,\al}\tp{\T}\ni\tp{\upkappa,\psi}\mapsto\tp{\mathsf{X}_{\upkappa,\psi} \pm \ii\mathsf{Y}_{\upkappa,\psi}},\;\tp{\mathsf{X}_{\upkappa,\psi} \pm \ii\mathsf{Y}_{\upkappa,\psi}}^{-1}\in C^{\upiota,\al}\tp{\T;\C}
        \end{equation}
        are well-defined, continuously differentiable, and map bounded sets to bounded sets.
    \end{enumerate}
\end{prop}
\begin{proof}
    The first and third items follow from arguments similar to that of the proof of Lemma~\ref{lem on mapping properties of the composition operators}; we omit the repetitive details for brevity. Thus, we shall only focus on the second item.

    We give an explicit presentation of the operator $\bf{Q}$. To begin, we define the Fourier multiplication operators $\mathbb{S}^\downarrow$ and $\mathbb{S}^{\uparrow}$ via
    \begin{multline}
        \mathscr{F}[\mathbb{S}^{\downarrow}\phi]\tp{\xi} = 2\pi|\xi|\tp{\tanh\tp{2\pi|\xi|h} - 1}\mathscr{F}[\phi]\tp{\xi}\\\text{and }\mathscr{F}[\mathbb{S}^{\uparrow}\phi]\tp{\xi} = \begin{cases}
            2\pi|\xi|\tp{\coth\tp{2\pi|\xi|h} - 1}\mathscr{F}[\phi]\tp{\xi}&\text{if }\xi\neq0,\\
            \mathscr{F}[\phi]\tp{0}&\text{if }\xi = 0,
        \end{cases}
    \end{multline}
    for $\xi\in\Z$ and suitable functions $\phi$. Notice that the symbols above decay exponentially and so the operators $\mathbb{S}^{\downarrow}$ and $\mathbb{S}^{\uparrow}$ are smoothing. The purpose of their introduction is the following pair of useful decomposition identities:
    \begin{equation}\label{multiplier_decomposition_0}
        G = \mathbb{S}^{\downarrow} - \ii\pd_1\mathbb{P}_+ + \ii\pd_1\mathbb{P}_-\text{ and }\m{Tr}_{\T\times\tcb{h}}\pd_1R_\psi = e_1 + \tp{-\ii\pd_1\mathbb{P}_+ + \ii\pd_1\mathbb{P}_-}\psi e_1 + \mathbb{S}^\uparrow\psi e_1 + \pd_1\psi e_2,
    \end{equation}
    where $G$ is the Dirichlet-to-Neumann operator \eqref{form:DN}.
    
     In the remainder of this proof, $\mathsf{f}\circ R_{\psi}$ is evaluated at $z=h$. Using the aforementioned operators and decompositions along with $[\cdot,\cdot]$ to denote the commutator bracket, we equate
    \begin{multline}\label{decomposition_1}
        \tp{\mathsf{f}\circ R_{\psi}}\cdot\tp{\pd_1 R_{\psi}}^\perp = -\mathbb{P}_+\tp{\tp{\mathsf{f}_1 + \ii \mathsf{f}_2}\circ R_{\psi}\mathbb{P}_+\pd_1\psi} -\mathbb{P}_-\tp{\tp{\mathsf{f}_1 - \ii \mathsf{f}_2}\circ R_{\psi}\mathbb{P}_-\pd_1\psi} \\ +[\tp{\mathsf{f}_1 + \ii \mathsf{f}_2}\circ R_{\psi},\mathbb{P}_-]\pd_1\mathbb{P}_+\psi + [\tp{\mathsf{f}_1 - \ii \mathsf{f}_2}\circ R_{\psi},\mathbb{P}_+]\pd_1\mathbb{P}_-\psi+\tp{\mathsf{f}_2\circ R_{\psi}}\tp{1 + \mathbb{S}^\uparrow\psi}\\-\int_{\T}\tp{\mathsf{f}_1\circ R_{\psi}}\pd_1\psi + \ii\int_{\T}\tp{\mathsf{f}_2\circ R_{\psi}}\tp{\mathbb{P}_--\mathbb{P}_+}\pd_1\psi
    \end{multline}
    and 
    \begin{multline}\label{decomposition_2}
        G\tp{\upphi\circ R_{\psi}} = -\mathbb{P}_+\tp{\tp{\pd_1\upphi + \ii\pd_2\upphi}\circ R_{\psi}\mathbb{P}_+\pd_1\psi}  -\mathbb{P}_-\tp{\tp{\pd_1\upphi - \ii\pd_2\upphi}\circ R_{\psi}\mathbb{P}_-\pd_1\psi} \\
        + [\mathbb{P}_-,\tp{\pd_1\upphi + \ii\pd_2\upphi}\circ R_{\psi}]\pd_1\mathbb{P}_+\psi
        +[\mathbb{P}_+,\tp{\pd_1\upphi - \ii\pd_2\upphi}\circ R_{\psi}]\pd_1\mathbb{P}_-\psi \\+ \mathbb{S}^{\downarrow}\tp{\upphi\circ R_{\psi}} + \ii\tp{\mathbb{P}_- - \mathbb{P}_+}\tp{\tp{\pd_1\upphi\circ R_{\psi}}\tp{1 + \mathbb{S}^\uparrow\psi}}.
    \end{multline}
    The first two terms on the right hand sides of~\eqref{decomposition_1} and~\eqref{decomposition_2} are the highest order contributions whereas the remaining terms are lower order. Substitution of~\eqref{multiplier_decomposition_0}, \eqref{decomposition_1}, and~\eqref{decomposition_2} into the definition of $\bf{P}$ in identity~\eqref{formal_operator_encoding} gives one the formula
    \begin{equation}\label{not_quite_right}
        \bf{P}\tp{\upkappa,\psi} =\P_+\tp{\tp{\mathsf{X}_{\upkappa,\psi} - \ii\mathsf{Y}_{\upkappa,\psi}}\P_+\pd_1\psi} + \P_-\tp{\tp{\mathsf{X}_{\upkappa,\psi} + \ii\mathsf{Y}_{\upkappa,\psi}}\P_-\pd_1\psi} + \bf{Q}\tp{\upkappa,\psi}
    \end{equation}
    with the operator $\bf{Q}$ explicitly written as
    \begin{multline}\label{the remainder operator Q}
        \bf{Q}\tp{\upkappa,\psi} = \upkappa^2[\mathbb{P}_-,\tp{\mathsf{f}_1 + \pd_1\upphi + \ii\tp{\mathsf{f}_2 + \pd_2\upphi}}\circ R_{\psi}]\pd_1\mathbb{P}_+\psi + \upkappa^2[\mathbb{P}_+,\tp{\mathsf{f}_1 + \pd_1\upphi - \ii\tp{\mathsf{f}_2 + \pd_2\upphi}}\circ R_{\psi}]\pd_1\mathbb{P}_-\psi \\ - \upkappa^2\tp{\mathsf{f}_2\circ R_{\psi}}\tp{1 + \mathbb{S}^\uparrow\psi} + \kappa^2\mathbb{S}^{\downarrow}\tp{\upphi\circ R_{\psi}} + \ii\upkappa^2\tp{\mathbb{P}_- - \mathbb{P}_+}\tp{\tp{\pd_1\upphi\circ R_{\psi}}\tp{1 + \mathbb{S}^\uparrow\psi}} + \m{g}\mathbb{S}^{\downarrow}\psi \\
        + \upkappa^2\int_{\T}\tp{\mathsf{f}_1\circ R_{\psi}}\pd_1\psi - \ii\upkappa^2\int_{\T}\tp{\mathsf{f}_2\circ R_{\psi}}\tp{\mathbb{P}_--\mathbb{P}_+}\pd_1\psi+ \upkappa^2\bf{K}\tp{\psi},
    \end{multline}
    where $\bf{K}:\z{C}^{1,\al}\tp{\T}\to C^{1,\al}\tp{\T}$ is the lower order operator introduced in~\eqref{defn_l_o_t}.

    To acquire the desired identity~\eqref{decomposing_identity_stinky} from~\eqref{not_quite_right} and deduce that $\bf{Q}$ is indeed $\R$-valued, we observe
    \begin{equation}\label{the important imaginary part identity}
        \m{Im}\tsb{\P_+\tp{\tp{\mathsf{X}_{\upkappa,\psi} - \ii\mathsf{Y}_{\upkappa,\psi}}\P_+\pd_1\psi} + \P_-\tp{\tp{\mathsf{X}_{\upkappa,\psi} + \ii\mathsf{Y}_{\upkappa,\psi}}\P_-\pd_1\psi}} = 0.
    \end{equation}
    Indeed, identity~\eqref{the important imaginary part identity} can be directly checked by showing the Fourier transform of the argument of the above imaginary part, call it $\phi$, satisfies the reality preservation condition $\mathscr{F}[\phi]\tp{-\xi} = \Bar{\mathscr{F}[\phi]\tp{\xi}}$, with the overline bar denoting complex conjugation.

    The final tasks of the proof are to check that $\bf{Q}$ as defined in~\eqref{the remainder operator Q} is continuously differentiable and compact as a mapping between the stated spaces. Continuous differentiability is a consequence of composing the analysis of Lemmas~\ref{lem on mapping properties of the composition operators} and~\ref{lem on analysis of the lower order bulk remainder} with simple product-type nonlinearities; in particular, there is no need to actually exploit any of the commutator cancellations.

    Compactness of $\bf{Q}$, on the other hand, does additionally require the commutator bounds of Proposition~\ref{prop on Pplus commutators}; specifically, we need the second estimate of~\eqref{comu_3}. But, once we are armed with this tool, we deduce actually that as the mapping into the smaller space
    \begin{equation}
    \R\times\z{C}^{1,\al}\tp{\T}\ni\tp{\upkappa,\psi}\mapsto\bf{Q}\tp{\upkappa,\psi}\in\z{C}^{1+\al}\tp{\T}
    \end{equation}
    the operator $\bf{Q}$ remains bounded on bounded sets. The compactness of $\bf{Q}$ as a mapping into $\z{C}^\al\tp{\T}$ is now a consequence of compact embedding $\z{C}^{1+\al}\tp{\T}\emb\z{C}^\al\tp{\T}$.
\end{proof}

The last result of this subsection records finer estimates available on the constituents of the decomposition of equation~\eqref{decomposing_identity_stinky}.

\begin{prop}[More precise operator decomposition estimates]\label{proposition on more precise operator decomposition estimates}
    Let $\al$, $\mathsf{f}$, $\upphi$, $\m{c}$, and $\m{g}$ be as in Lemma~\ref{lem on operator encoding} and let $0<\ep\le1/2$, $\ep<\del$, $0<\alpha<1$. The following hold.
    \begin{enumerate}
        \item The map $\bf{Q}$ from the second item of Proposition~\ref{prop on operator decomposition} uniquely extends to well-defined continuous functions
        \begin{equation}\label{me_feel}
            \bf{Q}:\R\times\z{H}^{1/2 + \del}\tp{\T}\to W^{\ep,1/\tp{1-\ep}}\tp{\T}\text{ and }\bf{Q}:\R\times\z{H}^{1 + \al/2}\tp{\T}\to  C^\al\tp{\T}.
        \end{equation}
        that send bounded sets to bounded sets.
        \item Assume that one of the ellipticity conditions of equation~\eqref{conditions for the hidden ellipticity} holds. Then the maps
        \begin{equation}
            \R\times\z{H}^{1/2 + \ep}\tp{\T}\ni\tp{\upkappa,\psi}\mapsto\tp{\mathsf{X}_{\upkappa,\psi}\pm\ii\mathsf{Y}_{\upkappa,\psi}},\;\tp{\mathsf{X}_{\upkappa,\psi}\pm\ii\mathsf{Y}_{\upkappa,\psi}}^{-1}\in H^{1/2 + \ep}\tp{\T;\C}
        \end{equation}
        are well-defined and send bounded sets to bounded sets.
    \end{enumerate}
\end{prop}
\begin{proof}
    The second item follows from arguments similar to that of the proof of Lemma~\ref{lem on mapping properties of the composition operators}. We omit the repetitive details.

    The first item is a direct calculation on the expression for $\bf{Q}$ from equation~\eqref{the remainder operator Q} that uses Lemmas~\ref{lem on mapping properties of the composition operators} and~\ref{lem on analysis of the lower order bulk remainder}, Proposition~\ref{prop on Pplus commutators}, the embedding $H^{1+\alpha/2}(\T)\emb C^\alpha(\T) $, and finally bounds of the type
    \begin{equation}
        \babs{\int_{\T}f^1\pd_1f^2}\lesssim\tnorm{f^1}_{H^{1/2}\tp{\T}}\tnorm{f^2}_{H^{1/2}\tp{\T}}.
    \end{equation}
    We again elect to omit the routine verification.
\end{proof}

\section{Analysis of large solutions to the nonlinear operator equation}\label{section on analysis of the nonlinear operator}

Armed with the preliminary analysis of Section~\ref{subsection on preliminaries}, we can now begin the production of large solutions to the equation~\eqref{reformulation for the coordinate function psi}. The strategy is to reformulate the operator equation as a nonlinear compact perturbation of the identity and import Leray-Schauder degree theory manifesting as a global implicit function theorem. This reformulation is possible when we assume the satisfaction of one of the ellipticity conditions of~\eqref{conditions for the hidden ellipticity}. After this initial construction of the solutions, we then determine the weaker  norms in which they are large.

\subsection{Construction of the connected set of solutions}\label{subsection on the construction of the connected set of solutions}

Recall that in Proposition~\ref{prop on operator decomposition} we decomposed our main nonlinear operator $\bf{P}$ into the sum of a minimal principal part and a compact lower order remainder $\bf{Q}$. Our next result establishes some crucial properties for this principal part.

\begin{prop}[Properties of the principal part]\label{prop on properties of the principal part}
    Let $\al$, $\mathsf{f}$, $\upphi$, $\m{c}$, and $\m{g}$ be as in Lemma~\ref{lem on operator encoding} and assume one of the ellipticity conditions~\eqref{conditions for the hidden ellipticity} is satisfied. The following hold.
    \begin{enumerate}
        \item The operator valued function $\bf{A}:\R\times\z{C}^{\al}\tp{\T}\to\mathcal{L}\tp{\z{C}^{1,\al}\tp{\T};\z{C}^\al\tp{\T}}$, with action
        \begin{equation}\label{the_main_part_A}
            \bf{A}\tp{\upkappa,\psi}\Psi = \m{Re}\tsb{\P_+\tp{\tp{\mathsf{X}_{\upkappa,\psi} - \ii\mathsf{Y}_{\upkappa,\psi}}\P_+\pd_1\Psi} + \P_-\tp{\tp{\mathsf{X}_{\upkappa,\psi} + \ii\mathsf{Y}_{\upkappa,\psi}}\P_-\pd_1\Psi}}
        \end{equation}
        for all $\tp{\upkappa,\psi}\in\R\times\z{C}^\al\tp{\T}$ and $\Psi\in\z{C}^{1,\al}\tp{\T}$, is well-defined and maps bounded sets to bounded sets, and is continuously differentiable.
        \item The map $\bf{A}$ introduced in the previous item  takes values in the set of invertible operators and the family of inverses
        \begin{equation}
            \R\times\z{C}^\al\tp{\T}\ni\tp{\upkappa,\psi}\mapsto\tsb{\bf{A}\tp{\upkappa,\psi}}^{-1}\in\mathcal{L}\tp{\z{C}^\al\tp{\T};\z{C}^{1,\al}\tp{\T}}
        \end{equation}
        map bounded sets to bounded sets and are continuously differentiable.
        \item The map $\bf{A}$ satisfies for all $\tp{\upkappa,\psi}\in\R\times\z{C}^{1,\al}\tp{\T}$ the equality
        \begin{equation}
            \bf{P}\tp{\upkappa,\psi} = \bf{A}\tp{\upkappa,\psi}\psi + \bf{Q}\tp{\upkappa,\psi}
        \end{equation}
        where $\bf{Q}$ is defined in~\eqref{the remainder operator Q}; hence, for all such $\tp{\upkappa,\psi}$ we have that $\bf{P}\tp{\upkappa,\psi} = 0$ if and only if $\psi = -\tsb{\bf{A}\tp{\upkappa,\psi}}^{-1}\bf{Q}\tp{\upkappa,\psi}$.
    \end{enumerate}
\end{prop}
\begin{proof}
    The first item is a direct computation and, assuming the second item is true, so too is the third item. Thus, we only focus on the proof of the second item.

    We begin with the proof of the following auxiliary claim. For each $1\le\lambda<\infty$ there exists $C_\lambda<\infty$ with the property that for all $\tp{\upkappa,\psi}\in\R\times\z{C}^{\al}\tp{\T}$ satisfying $|\upkappa| + \tnorm{\psi}_{C^\al\tp{\T}}\le\lambda$ and all $\Psi\in\z C^{1,\al}\tp{\T}$ we have the closed range estimate
    \begin{equation}\label{closed range estimates}
        \tnorm{\Psi}_{C^{1,\al}\tp{\T}}\le C_\lambda\tnorm{\bf{A}\tp{\upkappa,\psi}\Psi}_{C^\al\tp{\T}}.
    \end{equation}
    
    As a first step to showing estimate~\eqref{closed range estimates} to be true, we derive a Sobolev space variant. By orthogonality and reality considerations, we have the equalities
    \begin{equation}\label{equality_1}
        -\m{Re}\bsb{\int_{\T}\bf{A}\tp{\upkappa,\psi}\Psi\Bar{\tp{\P_+ + \P_-}\pd_1\Psi}} = -\int_{\T}\mathsf{X}_{\upkappa,\psi}\tp{\tabs{\P_+\pd_1\Psi}^2 + \tabs{\P_-\pd_1\Psi}^2}
    \end{equation}
    and
    \begin{equation}\label{equality_2}
        -\m{Im}\bsb{\int_{\T}\bf{A}\tp{\upkappa,\psi}\Psi\Bar{\tp{\P_+ - \P_-}\pd_1\Psi}} = \int_{\T}\mathsf{Y}_{\upkappa,\psi}\tp{\tabs{\P_+\pd_1\Psi}^2 + \tabs{\P_-\pd_1\Psi}^2}.
    \end{equation}
    The satisfaction of the ellipticity conditions~\eqref{conditions for the hidden ellipticity} means that either $\tabs{\mathsf{X}_{\upkappa,\psi}}\ge\tabs{\m{c}}>0$ or $\mathsf{Y}_{\upkappa,\psi}\ge\m{g}>0$. In either case, we can use the Cauchy-Schwarz inequality and identities~\eqref{equality_1} and~\eqref{equality_2} to deduce that the following weaker version of~\eqref{closed range estimates} holds
    \begin{equation}\label{easy closed range estimate}
        \tnorm{\Psi}_{H^1\tp{\T}}\le C_\lambda\tnorm{\bf{A}\tp{\upkappa,\psi}\Psi}_{L^2\tp{\T}}.
    \end{equation}

    We now shall use~\eqref{easy closed range estimate} with a parametrix and interpolation argument to boost up to the desired~\eqref{closed range estimates}. Since
    \begin{equation}
        \P_+\tp{\bf{A}\tp{\upkappa,\psi}\Psi} = \P_+\tp{\tp{\mathsf{X}_{\upkappa,\psi} - \ii\mathsf{Y}_{\upkappa,\psi}}\P_+\pd_1\Psi}
    \end{equation}
    we can introduce a commutator and multiply by $\tp{\mathsf{X}_{\upkappa,\psi} - \ii\mathsf{Y}_{\upkappa,\psi}}^{-1}$ to get
    \begin{equation}\label{I_E}
        \P_+\pd_1\Psi = \tp{\mathsf{X}_{\upkappa,\psi} - \ii\mathsf{Y}_{\upkappa,\psi}}^{-1}\tp{\P_+\tp{\bf{A}\tp{\upkappa,\psi}\Psi} + \tsb{\mathsf{X}_{\upkappa,\psi} - \ii\mathsf{Y}_{\upkappa,\psi},\P_+}\P_+\pd_1\Psi}.
    \end{equation}
    By taking the norm of~\eqref{I_E} in the space $C^\al\tp{\T}$ and using the third item of Proposition~\ref{prop on operator decomposition}, the commutator estimate~\eqref{comu_3} from Proposition~\ref{prop on Pplus commutators}, the boundedness of $\P_+$ on the H\"older spaces, and noting that $\Psi$ is $\R$-valued, we learn
    \begin{equation}\label{the intermediate estimate}
        \tnorm{\Psi}_{C^{1,\al}\tp{\T}}\lesssim_\lambda\tnorm{\bf{A}\tp{\upkappa,\psi}\Psi}_{C^\al\tp{\T}} + \tnorm{\Psi}_{C^1\tp{\T}}.
    \end{equation}

    We pair inequality~\eqref{the intermediate estimate} with the embedding $H^1\tp{\T}\emb C^{1/2}\tp{\T}$, the estimate~\eqref{easy closed range estimate}, and the interpolation bounds
    \begin{equation}
        \tnorm{\Psi}_{C^1\tp{\T}}\lesssim\tnorm{\Psi}_{C^{1/2}\tp{\T}}^{\f{2\al}{1 + 2\al}}\tnorm{\Psi}_{C^{1,\al}\tp{\T}}^{\f{1}{1+2\al}}
    \end{equation}
    to deduce
    \begin{equation}
        \tnorm{\Psi}_{C^{1,\al}\tp{\T}}\lesssim_\lambda\tnorm{\bf{A}\tp{\upkappa,\psi}\Psi}_{C^\al\tp{\T}} + \tp{\tnorm{\bf{A}\tp{\upkappa,\psi}\Psi}_{C^\al\tp{\T}}}^{\f{2\al}{1 + 2\al}}\tp{\tnorm{\Psi}_{C^{1,\al}\tp{\T}}}^{\f{1}{1 + 2\al}}.
    \end{equation}
    Estimate~\eqref{closed range estimates} now follows by Young's inequality and absorption.

    Upon having established the auxiliary claim, we realize that to complete the proof of the second item, it remains only to show that the pointwise inverse $\tsb{\bf{A}\tp{\upkappa,\psi}}^{-1}$ exists for all $\upkappa$ and $\psi$. Since we have the uniform closed ranges estimates~\eqref{closed range estimates}, the sough after invertibility will follow from the method of continuity, see, e.g., Theorem 4.51 in Abramovich and Aliprantis~\cite{MR1921782}, and the establishment of the existence of $\tsb{\bf{A}\tp{0,0}}^{-1}$. We compute explicitly that
    \begin{equation}
        \bf{A}\tp{0,0}\Psi = -\m{c}\pd_1\Psi - \ii\m{g}\tp{\P_+ - \P_-}\pd_1\Psi
    \end{equation}
    and so the inverse to $\bf{A}\tp{0,0}$ corresponds to the application of the Fourier multiplication operator with symbol
    \begin{equation}
        \Z\ni\xi\mapsto\mathds{1}_{\Z\setminus\tcb{0}}\tp{\xi}\tp{2\pi\m{g}|\xi| - 2\pi\ii\m{c}\xi}^{-1}\in\C.
    \end{equation}
    This inverse is well-defined since $\m{g}\m{c}\ne 0$ under \eqref{conditions for the hidden ellipticity}, and is bounded $\z{C}^\al\tp{\T}\to\z{C}^{1,\al}\tp{\T}$ as claimed.
\end{proof}

The third item of Proposition~\ref{prop on properties of the principal part} shows that solving the equation $\bf{P}\tp{\upkappa,\psi} = 0$ is equivalent to finding the zero set of a compact perturbation of the identity. This latter formulation is particularly convenient for initializing our construction and analysis of large solutions, as it allows us to import Leray-Schauder degree theory. The subsequent result explores this direction.

\begin{prop}[Existence of a global connected set of solutions]\label{prop on existence of a global connected set of solutions}
    Assume the hypotheses of Proposition~\ref{prop on properties of the principal part}. There exists a connected solution set $\mathcal{C}\subset\R\times\z{C}^{1,\al}\tp{\T}$ satisfying the following properties.
    \begin{enumerate}
        \item We have $\tp{0,0}\in\mathcal{C}$.
        \item For all $\tp{\upkappa,\psi}\in\mathcal{C}$ we have $\bf{P}\tp{\upkappa,\psi} = 0$.
        \item The function
        \begin{equation}\label{b_u_1}
            \mathcal{C}\ni\tp{\upkappa,\psi}\mapsto\tabs{\upkappa} + \tnorm{\psi}_{C^{1,\al}\tp{\T}}\in\R
        \end{equation}
        is unbounded.
    \end{enumerate}
\end{prop}
\begin{proof}
    We use the global implicit function theorem for compact perturbations of the identity based on Leray-Schauder degree theory as formulated in Theorem 4.6 in Nguyen~\cite{NGCD}; note that this is a special case of Theorem II.6.1 in Kielh\"{o}fer~\cite{MR2859263}.

    The map $\bf{F}:\R\times\z{C}^{1,\al}\tp{\R}\to\z{C}^{1,\al}\tp{\R}$ defined via
    \begin{equation}\label{the_map_is_a_gift}
        \bf{F}\tp{\upkappa,\psi} = \psi + \tsb{\bf{A}\tp{\upkappa,\psi}}^{-1}\bf{Q}\tp{\upkappa,\psi},\;\tp{\upkappa,\psi}\in\R\times\z{C}^{1,\al}\tp{\T},
    \end{equation}
    with $\bf{A}$ and $\bf{Q}$ as in equations~\eqref{the_main_part_A} and~\eqref{the remainder operator Q}, respectively, is well-defined and class $C^1$ in the Fr\'{e}chet sense as a consequence of Propositions~\ref{prop on operator decomposition} and~\ref{prop on properties of the principal part}. The third item of Proposition~\ref{prop on properties of the principal part} implies that the zero sets of $\bf{F}$ and $\bf{P}$ coincide. Moreover, the difference between $\bf{F}$ and the identity is compact, as it can be factored into the composition of the mapping 
    \begin{equation}
        \tp{\R\times\z{C}^\al\tp{\T}}\times\z{C}^\al\tp{\T}\ni\tp{\tp{\upkappa,\psi},\phi}\mapsto\tsb{\bf{A}\tp{\upkappa,\psi}}^{-1}\phi\in\z{C}^{1,\al}\tp{\T},
    \end{equation}
    that is bounded on bounded sets, and the compact mapping
    \begin{equation}
        \R\times\z{C}^{1,\al}\tp{\T}\ni\tp{\upkappa,\psi}\mapsto\tp{\tp{\upkappa,\psi},\bf{Q}\tp{\upkappa,\psi}}\in\tp{\R\times\z{C}^{\al}\tp{\T}}\times\z{C}^\al\tp{\T}.
    \end{equation}
    
    We also observe that the partial derivative of $\bf{F}$ at the trivial solution is an isomorphism. Indeed, a direct computation unveils the identity
    \begin{equation}\label{CHAIT}
        D_2\bf{F}\tp{0,0}\Psi = \Psi + \tsb{\bf{A}\tp{0,0}}^{-1}D_2\bf{Q}\tp{0,0}\Psi
    \end{equation}
    and hence, by taking Fourier transforms,
    \begin{equation}\label{TCHAI}
        \mathscr{F}\tsb{D_2\bf{F}\tp{0,0}\Psi}\tp{\xi} = \upmu(\xi)\mathscr{F}[\Psi]\tp{\xi}\text{ for }\upmu(\xi) = \mathds{1}_{\Z\setminus\tcb{0}}\tp{\xi}\bp{1 + \f{2\pi\m{g}|\xi|\tp{\tanh\tp{2\pi|\xi|h} - 1}}{2\pi\m{g}|\xi| - 2\pi\ii\m{c}\xi}},\;\xi\in\Z.
    \end{equation}
    The symbol $\upmu$ in~\eqref{TCHAI} is bounded above and below in modulus by positive constants (uniformly for $\xi\neq0$) thanks to the satisfaction of condition~\eqref{conditions for the hidden ellipticity}; hence, $\m{ker}D_2\bf{F}(0,0) = \tcb{0}$. Since $D_2F(0,0)$ is Fredholm with vanishing index, being a compact perturbation of the identity, we conclude that $D_2\bf{F}\tp{0,0}$ is invertible.

    The hypotheses of the global implicit function theorem as formulated in Theorem 4.6 in~\cite{NGCD} are thus satisfied by the map $\bf{F}$~\eqref{the_map_is_a_gift}. Therefore, if we let 
    \begin{equation}
        \mathcal{S} = \m{cl}\tcb{\tp{\upkappa,\psi}\in\R\times\z{C}^{1,\al}\tp{\T}\;:\;\bf{F}\tp{\upkappa,\psi} = 0}
    \end{equation}
    denote the closure of the solution set $\tcb{\bf{F} = 0}$ and let $\mathcal{C}\subseteq\mathcal{S}$ denote the connected component of $\mathcal{S}$ that contains $\tp{0,0}$, then either $\mathcal{C}$ is unbounded or $\mathcal{C}\setminus\tcb{\tp{0,0}}$ is connected.

    The proof is complete as soon as we show that $\mathcal{C}\setminus\tcb{\tp{0,0}}$ is disconnected. In fact, we claim that
    \begin{equation}\label{the disconnected claim}
        \mathcal{C}\setminus\tcb{\tp{0,0}}\subset\tp{\R\setminus\tcb{0}}\times\z{C}^{1,\al}\tp{\T},
    \end{equation}
    which, in particular, shows that $\mathcal{C}\setminus\tcb{\tp{0,0}}$ is disconnected. We prove the contrapositive of~\eqref{the disconnected claim}. Suppose that $\tp{\upkappa,\psi}\in\tp{\tcb{0}\times\z{C}^{1,\al}\tp{\T}}\cap\mathcal{C}$ so that $\bf{F}\tp{0,\psi} = 0$. Then, $\psi + \tsb{\bf{A}\tp{0,0}}^{-1}\bf{Q}\tp{0,0}\psi = 0$. In turn, for $\upmu$ as in~\eqref{TCHAI},
    \begin{equation}
        \forall\;\xi\in\Z,\;\upmu\tp{\xi}\mathscr{F}\tsb{\psi}\tp{\xi} = 0\text{ and hence }\psi = 0.
    \end{equation}
    So indeed the claimed inclusion~\eqref{the disconnected claim} holds.
\end{proof}

\subsection{Improvements on the blow-up quantity}\label{subsection on blow up refinements}

In Proposition~\ref{prop on existence of a global connected set of solutions} we have constructed a connected set of solutions $\mathcal{C}$ to the equation $\bf{P} = 0$, where $\bf{P}$ is the operator from~\eqref{formal_operator_encoding}, on which the function~\eqref{b_u_1} is unbounded. The goal of this subsection is to sharpen as much as we can the quantity that is blowing up by showing that norms of much lower regularity are actually unboundedly large on the set $\mathcal{C}$.

\begin{prop}[Blow-up refinement, I]\label{prop on blow up refinement I}
    Assume the hypotheses of Proposition~\ref{prop on existence of a global connected set of solutions} and let $\mathcal{C}$ be the connected set of solutions constructed therein. For every $0<\ep\le1/2$ the function 
    \begin{equation}\label{the_blow_up_1}
        \mathcal{C}\ni\tp{\upkappa,\psi}\mapsto\tabs{\upkappa} + \tnorm{\psi}_{H^{1/2 + \ep}\tp{\T}}\in\R
    \end{equation}
    is unbounded.
\end{prop}
\begin{proof}
    We begin by proving an auxiliary estimate. We claim that for all $0<\be<\del\le1/2$ and $M\in\R^+$ there exists $N\in\R^+$, with the property that for all $\tp{\upkappa,\psi}\in\R\times\z{C}^{1,\al}\tp{\T}$ satisfying $\bf{P}\tp{\upkappa,\psi} = 0$ and
    \begin{equation}\label{__if__}
        |\upkappa| + \tnorm{\psi}_{H^{1/2 + \del}\tp{\T}}\le M
    \end{equation}
    we have that
    \begin{equation}\label{_then_}
        |\upkappa| + \tnorm{\psi}_{H^{1/2 + 2\be}\tp{\T}}\le N.
    \end{equation}

    So let us assume that $\tp{\upkappa,\psi}\in\R\times\z{C}^{1,\al}\tp{\T}$ satisfy~\eqref{__if__} and $\bf{P}\tp{\upkappa,\psi} = 0$. We use identity~\eqref{I_E} with $\Psi = \psi$ and $\bf{A}\tp{\kappa,\psi}\psi = -\bf{Q}\tp{\upkappa,\psi}$ to see that
    \begin{equation}\label{identity_that_we_use}
        \P_+\pd_1\psi = \tp{\mathsf{X}_{\upkappa,\psi} - \ii\mathsf{Y}_{\upkappa,\psi}}^{-1}\tp{[\mathsf{X}_{\upkappa,\psi} - \ii\mathsf{Y}_{\upkappa,\psi},\P_+]\pd_1\P_+\psi - \P_+\bf{Q}\tp{\upkappa,\psi}}.
    \end{equation}
    The next step is to take the norm of identity~\eqref{identity_that_we_use} in the space $W^{\be,1/(1-\be)}\tp{\T;\C}$. Thanks to the reality of $\psi$ and to the continuity of the embedding and product maps
    \begin{equation}
        W^{1+\be,1/(1-\be)}\tp{\T;\C}\emb H^{1/2+2\be}\tp{\T;\C}\text{ and }H^{1/2 + \be}\tp{\T;\C}\times W^{\be,1/\tp{1-\be}}\tp{\T;\C}\to W^{\be,1/(1-\be)}\tp{\T;\C}
    \end{equation}
    we get
    \begin{multline}
        \tnorm{\psi}_{H^{1/2 + 2\be}\tp{\T}}\lesssim\tnorm{\P_+\pd_1\psi}_{W^{1+\be,1/(1-\be)}\tp{\T}}\\\lesssim\tnorm{\tp{\mathsf{X}_{\upkappa,\psi} - \ii\mathsf{Y}_{\upkappa,\psi}}^{-1}}_{H^{1/2 + \be}\tp{\T}}\tp{\tnorm{[\mathsf{X}_{\upkappa,\psi} - \ii\mathsf{Y}_{\upkappa,\psi},\P_+]\pd_1\P_+\Psi}_{W^{\be,1/(1-\be)}\tp{\T}} + \tnorm{\bf{Q}\tp{\upkappa,\psi}}_{W^{\be,1/(1-\be)}\tp{\T}}}.
    \end{multline}
    Now we use the assumed bound~\eqref{__if__} along with the first and second items of Proposition~\ref{proposition on more precise operator decomposition estimates} and the commutator estimate~\eqref{comu_1} from Proposition~\ref{prop on Pplus commutators} to deduce that the desired estimate~\eqref{_then_} indeed holds with some $N$ depending only on $M$, $\be$, and $\del$.

    Next, we shall prove that the function~\eqref{the_blow_up_1} is unbounded. Let us assume, for the sake of reaching a contradiction, that there exists $0<\ep\le1/2$ such that
    \begin{equation}\label{contra_hypo}
        \sup_{\tp{\upkappa,\psi}\in\mathcal{C}}\tp{|\upkappa| + \tnorm{\psi}_{H^{1/2 + \ep}\tp{\T}}}<\infty.
    \end{equation}
    By repeated application of the auxiliary estimate established in the first part of the proof, we deduce that for all $0\le s<3/2$
    \begin{equation}\label{intermediate_contradiction_step}
        \sup_{\tp{\upkappa,\psi}\in\mathcal{C}}\tp{|\upkappa| + \tnorm{\psi}_{H^{s}\tp{\T}}}<\infty.
    \end{equation}
    
    We then return to identity~\eqref{identity_that_we_use}, which continues to hold for $\tp{\upkappa,\psi}\in\mathcal{C}$. We now take the norm in $C^\al\tp{\T;\C}$. After using estimates~\eqref{you_make}, \eqref{comu_3}, and~\eqref{me_feel} along with~\eqref{intermediate_contradiction_step} for $s = 1+\al/2$, we find that
    \begin{equation}
        \sup_{\tp{\upkappa,\psi}\in\mathcal{C}}\tp{\tabs{\upkappa} + \tnorm{\psi}_{C^{1,\al}\tp{\T}}}<\infty,
    \end{equation}
    which is a contradiction of the unboundedness of the map~\eqref{b_u_1} from Proposition~\ref{prop on existence of a global connected set of solutions}.
\end{proof}

We now argue that within a restricted class of forcing data profiles, the blow-up quantity can be dropped even lower.

\begin{prop}[Blow-up refinement, II]\label{prop on blow up refinement II}
    Assume the hypotheses of Proposition~\ref{prop on existence of a global connected set of solutions} and let $\mathcal{C}$ be the connected set of solutions constructed therein and suppose additionally that $\mathsf{f} = 0$. For every $0<\ep\le\min\tcb{2\al,1/2}$ the function
    \begin{equation}\label{the_blow_up_2}
        \mathcal{C}\ni\tp{\upkappa,\psi}\mapsto|\upkappa| + \tnorm{\psi}_{C^\ep\tp{\T}}\in\R
    \end{equation}
    is unbounded.
\end{prop}
\begin{proof}
    Let us argue by contradiction, and suppose that there exists $0<\ep\le1/2$ such that
    \begin{equation}\label{the_contra_hyp}
        \sup_{\tp{\upkappa,\psi}\in\mathcal{C}}\tp{\tabs{\upkappa} + \tnorm{\psi}_{C^\ep\tp{\T}}}<\infty.
    \end{equation}
    Let $\tp{\upkappa,\psi}\in\mathcal{C}$. Since $\mathsf{f} = 0$, the expression for $\bf{Q}$ from~\eqref{the remainder operator Q} simplifies considerably:
    \begin{multline}
        \bf{Q}\tp{\upkappa,\psi} = \upkappa^2[\mathbb{P}_-,\tp{\pd_1\upphi + \ii\pd_2\upphi}\circ R_{\psi}]\pd_1\mathbb{P}_+\psi + \upkappa^2[\mathbb{P}_+,\tp{\pd_1\upphi - \ii\pd_2\upphi}\circ R_{\psi}]\pd_1\mathbb{P}_-\psi \\ + \kappa^2\mathbb{S}^{\downarrow}\tp{\upphi\circ R_{\psi}} + \ii\upkappa^2\tp{\mathbb{P}_- - \mathbb{P}_+}\tp{\tp{\pd_1\upphi\circ R_{\psi}}\tp{1 + \mathbb{S}^\uparrow\psi}} + \m{g}\mathbb{S}^{\downarrow}\psi.
    \end{multline}
    Hence, by using the commutator estimate~\eqref{comu_3} from Proposition~\ref{prop on Pplus commutators} (with $\be = 1-\ep/2$ and $\gam = \ep$) along with the uniform bounds~\eqref{the_contra_hyp} and Lemma~\ref{lem on mapping properties of the composition operators}, we find that
    \begin{equation}\label{good estimate on the Q dude}
        \sup_{\tp{\upkappa,\psi}\in\mathcal{C}}\f{\tnorm{\bf{Q}\tp{\upkappa,\psi}}_{C^{\ep/2}\tp{\T}}}{1 + \tnorm{\psi}_{C^{1-\ep/2}\tp{\T}}}<\infty.
    \end{equation}

    We now return to identity~\eqref{identity_that_we_use} for $\tp{\upkappa,\psi}\in\mathcal{C}$. By taking the norm in $C^{\ep/2}\tp{\T}$ and using again~\eqref{comu_3}, Lemma~\ref{lem on mapping properties of the composition operators}, \eqref{the_contra_hyp}, and~\eqref{good estimate on the Q dude} we get
    \begin{equation}
        \tnorm{\psi}_{C^{1,\ep/2}\tp{\T}}\lesssim\tnorm{\bf{Q}\tp{\upkappa,\psi}}_{C^{\ep/2}\tp{\T}} + \tnorm{\psi}_{C^{1-\ep/2}\tp{\T}}\lesssim 1 + \tnorm{\psi}_{C^{1 - \ep/2}\tp{\T}}
    \end{equation}
    for a finite implicit constant that is uniform in $\tp{\upkappa,\psi}\in\mathcal{C}$. Then, by the interpolation estimate
    \begin{equation}
        \tnorm{\psi}_{C^{1-\ep/2}\tp{\T}}\lesssim\tnorm{\psi}_{C^{\ep}\tp{\T}}^{2\ep/(2-\ep)}\tnorm{\psi}_{C^{1,\ep/2}\tp{\T}}^{(2-3\ep)/(2-\ep)},
    \end{equation}
    the bound~\eqref{the_contra_hyp}, and an absorption argument, we find that
    \begin{equation}
        \sup_{\tp{\upkappa,\psi}\in\mathcal{C}}\tp{|\upkappa| + \tnorm{\psi}_{C^{1,\ep/2}\tp{\T}}}<\infty.
    \end{equation}
    This contradicts Proposition~\ref{prop on blow up refinement I}.
\end{proof}

\section{Further refinements}\label{section on refinements}

\subsection{Conformal degeneracy}\label{subsection on blow up via coordinate degeneracy}

Combining Propositions~\ref{prop on blow up refinement I} and~\ref{prop on blow up refinement II} gets us most of the way to proving our first main result, Theorem~\ref{1_MAIN_THM}. To complete its justification, it remains to take the connected solution set $\mathcal{C}$ produced by Proposition~\ref{prop on existence of a global connected set of solutions} and study its intersection with the set $\mathcal{U}$, defined in equation~\eqref{the set of admissible domain coordinate functions}, of admissible coordinate functions. The content of this subsection is the following preliminary result that enumerates the essential properties of the set $\mathcal{U}$.
\begin{prop}[On the set of admissible coordinate functions]\label{prop on the set of admissible coordinate functions}
    The following hold where $\mathcal{U}$ is as defined in equation~\eqref{the set of admissible domain coordinate functions}.
    \begin{enumerate}
        \item The set $\mathcal{U}$ is open and $0\in\mathcal{U}$.
        \item For all $\psi\in\mathcal{U}$ we have $\mathscr{D}\tp{\psi}<\infty$, where $\mathscr{D}$ is the maximal distortion that is defined in equation~\eqref{the distortion}.
        \item For all $\psi\in\z{C}^{1,\al}\tp{\T}\cap\pd\mathcal{U}$ we have that
        \begin{equation}\label{exploding distortion near the boundary}
            \mathscr{D}\tp{\tilde{\psi}}\to\infty\text{ as }\mathcal{U}\ni\tilde{\psi}\to\psi.
        \end{equation}
    \end{enumerate}
\end{prop}
\begin{proof}
    We begin with the proof of the first item. Since $R_0\tp{w,z} = (w,z)$ for all $\tp{w,z}\in\T\times[0,h]$, we obviously have $0\in\mathcal{U}$. 

    The map 
    \begin{equation}
        \z{C}^{1,\al}\tp{\T}\ni\psi\mapsto\det\grad R_\psi\in C^0\tp{\T\times[0,h]}
    \end{equation}
    is continuous (thanks to Lemma~\ref{lemma on bounds on the Cauchy Riemann solver}) and so for each $\psi\in\mathcal{U}$ there exists $\epsilon_\psi>0$ such that 
    \begin{equation}
        \forall\;\tnorm{\psi - \tilde{\psi}}_{C^{1,\al}\tp{\T}}<\epsilon_\psi,\;\min_{\T\times[0,h]}\det\grad R_{\tilde{\psi}}>0.
    \end{equation}

    We now claim that for each $\psi\in\mathcal{U}$, there exists $c_\psi\in\R^+$ with the property that for all $p,q\in\T\times[0,h]$ with $p\neq q$ we have
    \begin{equation}\label{finite_distortion}
        \m{dist}_{\T\times[0,h]}\tp{p,q}\le c_\psi\m{dist}_{\T\times[0,h]}\tp{R_\psi\tp{p},R_{\psi}\tp{q}}.
    \end{equation}
    If this were false, then we could find convergent sequences $\tcb{p_n}_{n\in\N},\tcb{q_n}_{n\in\N}\subset\T\times[0,h]$ with $p_n\neq q_n$ for all $n\in\N$ and points $p,q\in\T\times[0,h]$ with $p_n\to p$ and $q_n\to q$ as $n\to\infty$ and
    \begin{equation}\label{setting_up_contra_}
        \f{1}{2^n}>\f{\m{dist}_{\T\times[0,h]}\tp{R_\psi\tp{p_n},R_\psi\tp{q_n}}}{\m{dist}_{\T\times[0,h]}\tp{p_n,q_n}}.
    \end{equation}
    If $p\neq q$, then passing to the limit as $n\to\infty$ in estimate~\eqref{setting_up_contra_} shows $R_\psi\tp{p} = R_\psi(q)$, which is impossible since $\psi\in\mathcal{U}$ implies that $R_\psi$ is injective. Thus it must be that $p = q$.

    For $n\in\N$ sufficiently large, we can use the fundamental theorem of calculus and the Cauchy-Riemann equations satisfied by $R_\psi$ to write
    \begin{equation}\label{this_will_be_reffed}
        \f{\m{dist}_{\T\times[0,h]}\tp{R_\psi\tp{p_n},R_\psi\tp{q_n}}}{\m{dist}_{\T\times[0,h]}\tp{p_n,q_n}} = \f{1}{\sqrt{2}}\babs{\int_0^1\grad R_\psi\tp{\tau p_n + (1-\tau)q_n}\;\m{d}\tau}<\f{1}{2^n}.
    \end{equation}
    Passing to the limit $n\to\infty$ in the above now shows that $\tabs{\grad R_\psi\tp{p}} = 0$. This is a contradiction of the fact that $\psi\in\mathcal{U}$ implies that $\det\grad R_\psi>0$ on the whole of $\T\times[0,h]$.

    So indeed the claimed reverse-Lipschitz bound of equation~\eqref{finite_distortion} must be true. In turn we find that, given $\psi\in\mathcal{U}$ and $c_\psi$ obeying~\eqref{finite_distortion}, there exists $\tilde{\epsilon}_\psi>0$ such that
    \begin{equation}
        \forall\;\tnorm{\psi - \tilde{\psi}}_{C^{1,\al}\tp{\T}}<\tilde{\epsilon}_\psi\;\text{and}\;\forall\;p,q\in\T\times[0,h],\;\m{dist}_{\T\times[0,h]}\tp{p,q}\le\tp{c_\psi/2}\m{dist}_{\T\times[0,h]}\tp{R_{\tilde{\psi}}\tp{p},R_{\tilde{\psi}}\tp{q}}.
    \end{equation}
    We may now conclude that $\mathcal{U}$ is indeed open, since we have shown that for all $\psi\in\mathcal{U}$ we have $\tilde{\psi}\in\mathcal{U}$ whenever $\tnorm{\tilde{\psi} - \psi}_{C^{1,\al}\tp{\T}}<\min\tcb{\epsilon_\psi,\tilde{\epsilon}_\psi}$.

    The second item is a special consequence of the preceding analysis since for all $\psi\in\mathcal{U}$ we have $\mathscr{D}\tp{\psi}\le c_\psi$ where $c_\psi$ is the finite positive constant as in~\eqref{finite_distortion}.

    Let us turn our attention to the third item. We shall argue by contradiction. If~\eqref{exploding distortion near the boundary} is false, then there would exist $\psi\in\pd\mathcal{U}$, $\tcb{\psi_n}_{n\in\N}\subset\mathcal{U}$, and $M<\infty$ such that
    \begin{equation}\label{not_another_one}
        \tnorm{\psi - \psi_n}_{C^{1,\al}\tp{\T}}\to0\text{ as }n\to\infty\text{ but }\sup_{n\in\N}\mathscr{D}\tp{\psi_n}\le M.
    \end{equation}
    We are going to associate with this sequence of real-valued functions a sequence of holomorphic $\C$-valued functions.

    Define the (doubled) complex domain
    \begin{equation}
        O = \tcb{\zeta\in\C\;:\;\tabs{\m{Im}\zeta}<h}
    \end{equation}
    and the sequence of functions $\tcb{f_n}_{n\in\N}$ via
    \begin{equation}
        f_n(\zeta) = \begin{cases}
            e_1\cdot R_{\psi_n}\tp{\m{Re}\zeta,\m{Im}\zeta} + \ii e_2\cdot R_{\psi_n}\tp{\m{Re}\zeta,\m{Im}\zeta}&\text{if }0\le\m{Im}\zeta<h,\\
            e_1\cdot R_{\psi_n}\tp{\m{Re}\zeta,-\m{Im}\zeta} - \ii e_2\cdot R_{\psi_n}\tp{\m{Re}\zeta,-\m{Im}\zeta}&\text{if }-h<\m{Im}h<\infty
        \end{cases}\text{ for all }\zeta\in O.
    \end{equation}
    Similarly, we can also define $f:O\to\C$ in terms of $R_\psi$ as above. 
    
    By virtue of the satisfaction of the Cauchy-Riemann equations by $\tcb{R_{\psi_n}}_{n\in\N}$, $R_{\psi}$ and the Schwarz reflection principle (see, e.g., Section 1 in Chapter IX of Conway~\cite{MR503901}), the maps $f$ and $\tcb{f_n}_{n\in\N}$ are holomorphic functions $O\to\C$. Moreover, these maps extend to class $C^{1,\al}$ functions on the closure $\Bar{O}$. Also, since $\tcb{\psi_n}_{n\in\N}\subset\mathcal{U}$, for each $n\in\N$ it follows that $f_n:\Bar{O}\to\C$ is injective and $\min_{\Bar{O}}\tabs{f_n'}>0$, where the prime denotes complex differentiation.

    The assumed convergence on the left hand side of~\eqref{not_another_one}, implies that $f_n\to f$ in $C^{1,\al}\tp{\Bar{O};\C}$ as well. We may then invoke variations on Hurwitz's theorem (see, e.g., Theorems 5 and 6 in Chapter 3, Section 2 of Narasimhan and Nievergelt~\cite{MR1803086}), after noting that $f'$ is not identically zero, to conclude that: 1) $f$ is injective in the interior region $O$ and $2)$ the complex derivative $f'$ is nowhere vanishing in $O$.

    We unpack these conclusions for $f$ to see what they imply about $R_\psi$. We deduce that: 1) $R_\psi:\T\times[0,h)\to\T\times[0,\infty)$ is injective and 2) $\det\grad R_\psi:\T\times[0,h)\to\R$ is everywhere positive. Since $\mathcal{U}$ is open and $f\in\pd\mathcal{U}$, we have that $f\not\in\mathcal{U}$. Synthesizing these observations, we conclude that one of the following cases must occur.
    \begin{itemize}
        \item There exists $w\in\T$ such that $\det\grad R_\psi\tp{w,h} = 0$.
        \item There exists $p\in\T\times\tcb{h}$ and $q\in\T\times[0,h]$ such that $R_\psi(p) = R_\psi(q)$
    \end{itemize}
   We can use the right hand side inequality of~\eqref{not_another_one} to eliminate the first case above. Indeed, for every $w_0,w_1\in\T$ and $n\in\N$ the contradiction hypothesis is that
   \begin{equation}\label{the_above}
       \m{dist}_{\T}\tp{w_0,w_1}\le M\m{dist}_{\T\times\tp{0,h}}\tp{R_{\psi_n}\tp{p,h},R_{\psi_n}\tp{q,h}}.
   \end{equation}
   Taking the limit $n\to\infty$ in~\eqref{the_above} implies that $|\pd_1 R_\psi(w,h)|\ge1/M$ for all $w\in\T$. So the first case does not occur since $\det\grad R_\psi = \tabs{\pd_1R_\psi}^2$; so $\det\grad R_\psi$ is everywhere positive on $\T\times\tcb{h}$ and hence on $\T\times[0,h]$.

   Thus we must have the existence of $p$ and $q$ as in the second item. It must be the case that $q\in\T\times\tcb{h}$ as well, as can be deduced by considering
   \begin{equation}\label{almost_done}
       R_\psi(q) = R_\psi(p) = \lim_{n\to\infty}R_{\psi_n}(p).
   \end{equation} 
   If $q\in\T\times[0,h)$, then $f_n$, for sufficiently large $n$, would, because of~\eqref{almost_done} and the fact that $f:O\to\C$ is injective with nowhere vanishing derivative, map $\pd O$ into the interior of $f_n(O)$, which cannot happen since $f_n$ extends to a $C^{1,\al}$ diffeomorphism of the closures. In turn we get
   \begin{equation}
       R_\psi(p_1,h) = R_\psi(q_1,h)
   \end{equation}
   which is also impossible given equation~\eqref{the_above}. Thus we have contradicted the assumption that the limit~\eqref{exploding distortion near the boundary} was false.
   \end{proof}

\subsection{Synthesis}\label{subsection on synthesis}

We may now combine Proposition~\ref{prop on the set of admissible coordinate functions} with the results of Section~\ref{section on analysis of the nonlinear operator} to complete the proof of our first main theorem.
\begin{proof}[Proof of Theorem~\ref{1_MAIN_THM}]
    Thanks to Propositions~\ref{prop on existence of a global connected set of solutions}, \ref{prop on blow up refinement I}, and~\ref{prop on blow up refinement II} we are granted a connected solution set $\tp{0,0}\in\mathcal{C}\subset\R\times\z{C}^{1,\al}\tp{\T}$ for the desired equation~\eqref{reformulation for the coordinate function psi} such that the function~\eqref{the_blow_up_1} is unbounded and, in the special case $\mathsf{f} = 0$, the function~\eqref{the_blow_up_2} is unbounded. We need only verify the blow-up of the quantities~\eqref{the blow up norm with distortion 1} or~\eqref{the blow up norm with distortion 2} on the restricted solution set $\mathcal{C}\cap\tp{\R\times\mathcal{U}}$, with $\mathcal{U}$ being defined in~\eqref{the set of admissible domain coordinate functions}.

    If $\mathcal{C}\subseteq\tp{\R\times\mathcal{U}}$, then of course there is nothing more to prove. Since $\mathcal{C}$ is connected, the only other case that can occur is that $\mathcal{C}\cap\tp{\R\times\pd\mathcal{U}}\neq\es$; more precisely, we must have $\Bar{\mathcal{C}\cap\tp{\R\times\mathcal{U}}}\cap\tp{\R\times\pd\mathcal{U}}\neq\es$. But then, by the third item of Proposition~\ref{prop on the set of admissible coordinate functions} we find that the function
    \begin{equation}
        \mathcal{C}\cap\tp{\R\times\mathcal{U}}\ni\tp{\upkappa,\psi}\mapsto\mathscr{D}\tp{\psi}\in\R
    \end{equation}
    is necessarily unbounded.
\end{proof}

Our second main theorem is proved by simply taking our solutions constructed to the fully nonlinear pseudodifferential equation~\eqref{reformulation for the coordinate function psi} and unpacking them to solutions to the parent free boundary problem~\eqref{traveling FBIDF}.

\begin{proof}[Proof of Theorem~\ref{2_MAIN_THM}]
    The mapping $R_\psi:\T\times(0,h)\to\T\times\R$ is smooth since the components of $R_\psi$ satisfy the Cauchy-Riemann equations as a consequence of Lemma~\ref{lemma on bounds on the Cauchy Riemann solver}. The same lemma also implies that $R_\psi:\T\times[0,h]\to\T\times\R$ is class $C^{1,\al}$. The inclusion of $\tp{\upkappa,\psi}\in\mathcal{C}\cap\tp{\R\times\mathcal{U}}$ and the definition of $\mathcal{U}$~\eqref{the set of admissible domain coordinate functions} completes the verification of the first item.

    The proof of the second item traverses in reverse Section~\ref{subsection on conformal and nonlocal reformulation}. Let us fix $\tp{\upkappa,\psi}\in\mathcal{C}\cap\tp{\R\times\mathcal{U}}$. We shall first construct the flattened pressure variable $q\in C^{1,\al}\tp{\T\times\tp{0,h}}$ as to solve system~\eqref{flattened formulation of the darcy flow}. Let us define $f^1\in C^{\al}\tp{\T\times\tp{0,h}}$, $f^2\in C^\al\tp{\T\times\tcb{0}}$, $f^3\in C^{1,\al}\tp{\T\times\tcb{h}}$ via
    \begin{equation}\label{New_York_City}
        f^1 = \tp{\upkappa^2/2}\tabs{\grad R_\psi}^2\tp{\grad\cdot\mathsf{f}}\circ R_\psi,\;f^2 = \upkappa^2\tp{f\circ R_\psi}\cdot\tp{\pd_1 R_\psi}^\perp,\text{ and }f^3=\m{g}\psi + \upkappa^2\tp{\upphi\circ R_\psi}.
    \end{equation}
    By standard Schauder theory there exits a unique function $q\in C^{1,\al}\tp{\T\times\tp{0,h}}$ that is a solution to the equations
    \begin{equation}\label{the schauder problem}
        \Delta q = f^1\text{ in }\T\times\tp{0,h},\;\pd_2q = f^2\text{ on }\T\times\tcb{0},\text{ and }q = f^3\text{ on }\T\times\tcb{h}.
    \end{equation}
    Now we use that $\bf{P}\tp{\upkappa,\psi} = 0$, with $\bf{P}$ as defined in~\eqref{formal_operator_encoding}, to check that the solution to~\eqref{the schauder problem} has the correct Neumann data on the top. Indeed, rearrangement of $\bf{P}\tp{\upkappa,\psi} = 0$ and substitution of definitions~\eqref{New_York_City} show that
    \begin{equation}
        \m{Tr}_{\T\times\tcb{h}}\pd_2q = G f^3 + S(f^1,f^2) = \m{c}\pd_1\psi + \upkappa^2\m{Tr}_{\T\times\tcb{h}}\tp{\mathsf{f}\circ R_\psi}\cdot\tp{\pd_1 R_\psi}^\perp,
    \end{equation}
    where we recall that $G$ and $S$ are defined in equation~\eqref{help_solve_1}. Thus we have a solution pair $\tp{\psi,q}$ to system~\eqref{flattened formulation of the darcy flow} with $\kappa = \upkappa^2$.

    Now we can straightforwardly define our solution to the equations of traveling free boundary incompressible Darcy flow~\eqref{traveling_FBIDF}. We let $\Omega_\psi = R_\psi\tp{\T\times\tp{0,h}}$ and $\m{p}\in C^{1,\al}\tp{\Omega_\psi}$ and $\m{v} \in C^\al\tp{\Omega_\psi;\R^2}$ be determined via
    \begin{equation}
        \m{p}(x) =  q\circ\tp{R_\psi}^{-1}(x) + \m{g}\tp{h - e_2\cdot x}\text{ and }\m{v}\tp{x} = \upkappa^2\mathsf{f}\tp{x} - \grad\m{p}\tp{x} - \m{g}e_2
    \end{equation}
    for all $x\in\Bar{\Omega_\psi}$. It is then a simple exercise in local elliptic regularity to see that $\m{p}$ and $\m{v}$ are class $C^4$ in the interior of their domain and solve the correct equations~\eqref{traveling_FBIDF} in $\Bar{\Omega_\psi}$, given that we already know $\tp{\psi,q}$ solve~\eqref{flattened formulation of the darcy flow} with $\kappa = \upkappa^2$.
\end{proof}

\section*{Acknowledgments}

The second author would like to thank Sijue Wu for helpful discussions on conformal flattening maps.

\bibliographystyle{abbrv}
\bibliography{bib.bib}

\begin{thebibliography}{10}

\bibitem{MR960687}
R.~Abraham, J.~E. Marsden, and T.~Ratiu.
\newblock {\em Manifolds, {T}ensor {A}nalysis, and {A}pplications}, volume~75 of {\em Applied Mathematical Sciences}.
\newblock Springer-Verlag, New York, second edition, 1988.

\bibitem{MR1921782}
Y.~A. Abramovich and C.~D. Aliprantis.
\newblock {\em An {I}nvitation to {O}perator {T}heory}, volume~50 of {\em Graduate Studies in Mathematics}.
\newblock American Mathematical Society, Providence, RI, 2002.

\bibitem{MR4097324}
T.~Alazard and O.~Lazar.
\newblock Paralinearization of the {M}uskat equation and application to the {C}auchy problem.
\newblock {\em Arch. Ration. Mech. Anal.}, 237(2):545--583, 2020.

\bibitem{MR4242131}
T.~Alazard and Q.-H. Nguyen.
\newblock On the {C}auchy problem for the {M}uskat equation. {II}: {C}ritical initial data.
\newblock {\em Ann. PDE}, 7(1):Paper No. 7, 25, 2021.

\bibitem{MR4541917}
T.~Alazard and Q.-H. Nguyen.
\newblock Endpoint {S}obolev theory for the {M}uskat equation.
\newblock {\em Comm. Math. Phys.}, 397(3):1043--1102, 2023.

\bibitem{MR1020348}
G.~Allaire.
\newblock Homogenization of the {S}tokes flow in a connected porous medium.
\newblock {\em Asymptotic Anal.}, 2(3):203--222, 1989.

\bibitem{MR2313156}
D.~M. Ambrose.
\newblock Well-posedness of two-phase {D}arcy flow in 3{D}.
\newblock {\em Quart. Appl. Math.}, 65(1):189--203, 2007.

\bibitem{MR3171344}
D.~M. Ambrose.
\newblock The zero surface tension limit of two-dimensional interfacial {D}arcy flow.
\newblock {\em J. Math. Fluid Mech.}, 16(1):105--143, 2014.

\bibitem{MR2768550}
H.~Bahouri, J.-Y. Chemin, and R.~Danchin.
\newblock {\em Fourier {A}nalysis and {N}onlinear {P}artial {D}ifferential {E}quations}, volume 343 of {\em Grundlehren der mathematischen Wissenschaften [Fundamental Principles of Mathematical Sciences]}.
\newblock Springer, Heidelberg, 2011.

\bibitem{BEAR}
J.~Bear.
\newblock {\em Dynamics of {F}luids in {P}orous {M}edia}.
\newblock Elsevier, New York, 1972.

\bibitem{MR3752178}
J.~Bear.
\newblock {\em Modeling {P}henomena of {F}low and {T}ransport in {P}orous {M}edia}, volume~31 of {\em Theory and Applications of Transport in Porous Media}.
\newblock Springer, Cham, 2018.
\newblock With appendices by Raphael Semiat, Jonathan Ajo-Franklin, Marco Voltolini and David Trebotich.

\bibitem{MR1926867}
G.~Bourdaud and M.~Lanza~de Cristoforis.
\newblock Functional calculus in {H}\"older-{Z}ygmund spaces.
\newblock {\em Trans. Amer. Math. Soc.}, 354(10):4109--4129, 2002.

\bibitem{MR2369144}
G.~Bourdaud and M.~Lanza~de Cristoforis.
\newblock Regularity of the symbolic calculus in {B}esov algebras.
\newblock {\em Studia Math.}, 184(3):271--298, 2008.

\bibitem{MR4797733}
J.~Brownfield and H.~Q. Nguyen.
\newblock Slowly traveling gravity waves for {D}arcy flow: existence and stability of large waves.
\newblock {\em Comm. Math. Phys.}, 405(10):Paper No. 222, 25, 2024.

\bibitem{MR1956130}
B.~Buffoni and J.~Toland.
\newblock {\em Analytic {T}heory of {G}lobal {B}ifurcation}.
\newblock Princeton Series in Applied Mathematics. Princeton University Press, Princeton, NJ, 2003.
\newblock An introduction.

\bibitem{MR3869383}
S.~Cameron.
\newblock Global well-posedness for the two-dimensional {M}uskat problem with slope less than 1.
\newblock {\em Anal. PDE}, 12(4):997--1022, 2019.

\bibitem{GWPMSSC}
S.~Cameron.
\newblock Global wellposedness for the 3{D} muskat problem with medium size slope.
\newblock 2020.
\newblock Preprint, \href{https://arxiv.org/abs/2002.00508}{arXiv:2002.00508}.

\bibitem{MR3048596}
A.~Castro, D.~C\'ordoba, C.~Fefferman, and F.~Gancedo.
\newblock Breakdown of smoothness for the {M}uskat problem.
\newblock {\em Arch. Ration. Mech. Anal.}, 208(3):805--909, 2013.

\bibitem{MR3519969}
A.~Castro, D.~C\'ordoba, C.~Fefferman, and F.~Gancedo.
\newblock Splash singularities for the one-phase {M}uskat problem in stable regimes.
\newblock {\em Arch. Ration. Mech. Anal.}, 222(1):213--243, 2016.

\bibitem{MR2993754}
A.~Castro, D.~C\'ordoba, C.~Fefferman, F.~Gancedo, and M.~L\'opez-Fern\'andez.
\newblock Rayleigh-{T}aylor breakdown for the {M}uskat problem with applications to water waves.
\newblock {\em Ann. of Math. (2)}, 175(2):909--948, 2012.

\bibitem{MR3765551}
R.~M. Chen, S.~Walsh, and M.~H. Wheeler.
\newblock Existence and qualitative theory for stratified solitary water waves.
\newblock {\em Ann. Inst. H. Poincar\'e{} C Anal. Non Lin\'eaire}, 35(2):517--576, 2018.

\bibitem{CWW}
R.~M. Chen, S.~Walsh, and M.~H. Wheeler.
\newblock Global bifurcation for monotone fronts of elliptic equations.
\newblock {\em Journal of the European Mathematical Society}, 2024.

\bibitem{MR3415681}
C.~H.~A. Cheng, R.~Granero-Belinch\'on, and S.~Shkoller.
\newblock Well-posedness of the {M}uskat problem with {$H^2$} initial data.
\newblock {\em Adv. Math.}, 286:32--104, 2016.

\bibitem{MR3595492}
P.~Constantin, D.~C\'ordoba, F.~Gancedo, L.~Rodr\'iguez-Piazza, and R.~M. Strain.
\newblock On the {M}uskat problem: global in time results in 2{D} and 3{D}.
\newblock {\em Amer. J. Math.}, 138(6):1455--1494, 2016.

\bibitem{MR2998834}
P.~Constantin, D.~C\'ordoba, F.~Gancedo, and R.~M. Strain.
\newblock On the global existence for the {M}uskat problem.
\newblock {\em J. Eur. Math. Soc. (JEMS)}, 15(1):201--227, 2013.

\bibitem{MR503901}
J.~B. Conway.
\newblock {\em Functions of {O}ne {C}omplex {V}ariable}, volume~11 of {\em Graduate Texts in Mathematics}.
\newblock Springer-Verlag, New York-Berlin, second edition, 1978.

\bibitem{MR2753607}
A.~C\'ordoba, D.~C\'ordoba, and F.~Gancedo.
\newblock Interface evolution: the {H}ele-{S}haw and {M}uskat problems.
\newblock {\em Ann. of Math. (2)}, 173(1):477--542, 2011.

\bibitem{MR3071395}
A.~C\'ordoba, D.~C\'ordoba, and F.~Gancedo.
\newblock Porous media: the {M}uskat problem in three dimensions.
\newblock {\em Anal. PDE}, 6(2):447--497, 2013.

\bibitem{MR2318314}
D.~C\'ordoba and F.~Gancedo.
\newblock Contour dynamics of incompressible 3-{D} fluids in a porous medium with different densities.
\newblock {\em Comm. Math. Phys.}, 273(2):445--471, 2007.

\bibitem{MR4363243}
D.~C\'ordoba and O.~Lazar.
\newblock Global well-posedness for the 2{D} stable {M}uskat problem in {$H^{3/2}$}.
\newblock {\em Ann. Sci. \'Ec. Norm. Sup\'er. (4)}, 54(5):1315--1351, 2021.

\bibitem{DARCY}
H.~Darcy.
\newblock {\em Les {F}ontaines {P}ubliques de la {V}ille de {D}ijon}.
\newblock Victor Dalmont, Paris, 1856.

\bibitem{MR1984157}
F.~Dias and G.~Iooss.
\newblock Water-waves as a spatial dynamical system.
\newblock In {\em Handbook of {M}athematical {F}luid {D}ynamics, {V}ol. {II}}, pages 443--499. North-Holland, Amsterdam, 2003.

\bibitem{MR4655356}
H.~Dong, F.~Gancedo, and H.~Q. Nguyen.
\newblock Global well-posedness for the one-phase {M}uskat problem.
\newblock {\em Comm. Pure Appl. Math.}, 76(12):3912--3967, 2023.

\bibitem{DGNGWP3D}
H.~Dong, F.~Gancedo, and H.~Q. Nguyen.
\newblock Global well-posedness for the one-phase {M}uskat problem in 3{D}.
\newblock 2023.
\newblock Preprint, \href{https://arxiv.org/abs/2308.14230}{arXiv:2308.14230}.

\bibitem{MR380547}
P.~Dr\'abek.
\newblock Continuity of {N}\u emyckij's operator in {H}\"older spaces.
\newblock {\em Comment. Math. Univ. Carolinae}, 16:37--57, 1975.

\bibitem{MR4227171}
P.~T. Flynn and H.~Q. Nguyen.
\newblock The vanishing surface tension limit of the {M}uskat problem.
\newblock {\em Comm. Math. Phys.}, 382(2):1205--1241, 2021.

\bibitem{MR3608884}
F.~Gancedo.
\newblock A survey for the {M}uskat problem and a new estimate.
\newblock {\em SeMA J.}, 74(1):21--35, 2017.

\bibitem{MR3899970}
F.~Gancedo, E.~Garc\'ia-Ju\'arez, N.~Patel, and R.~M. Strain.
\newblock On the {M}uskat problem with viscosity jump: global in time results.
\newblock {\em Adv. Math.}, 345:552--597, 2019.

\bibitem{MR4487512}
F.~Gancedo and O.~Lazar.
\newblock Global well-posedness for the three dimensional {M}uskat problem in the critical {S}obolev space.
\newblock {\em Arch. Ration. Mech. Anal.}, 246(1):141--207, 2022.

\bibitem{MR1814364}
D.~Gilbarg and N.~S. Trudinger.
\newblock {\em Elliptic {P}artial {D}ifferential {E}quations of {S}econd {O}rder}.
\newblock Classics in Mathematics. Springer-Verlag, Berlin, 2001.
\newblock Reprint of the 1998 edition.

\bibitem{MR3215843}
J.~G\'omez-Serrano and R.~Granero-Belinch\'on.
\newblock On turning waves for the inhomogeneous {M}uskat problem: a computer-assisted proof.
\newblock {\em Nonlinearity}, 27(6):1471--1498, 2014.

\bibitem{MR4064479}
R.~Granero-Belinch\'on and O.~Lazar.
\newblock Growth in the {M}uskat problem.
\newblock {\em Math. Model. Nat. Phenom.}, 15:Paper No. 7, 23, 2020.

\bibitem{Groves_2004}
M.~D. Groves.
\newblock {Steady water waves}.
\newblock {\em J. Nonlinear Math. Phys.}, 11(4):435--460, 2004.

\bibitem{MR4406719}
S.~V. Haziot, V.~M. Hur, W.~A. Strauss, J.~F. Toland, E.~Wahl\'{e}n, S.~Walsh, and M.~H. Wheeler.
\newblock Traveling water waves---the ebb and flow of two centuries.
\newblock {\em Quart. Appl. Math.}, 80(2):317--401, 2022.

\bibitem{MR2859263}
H.~Kielh\"ofer.
\newblock {\em Bifurcation {T}heory}, volume 156 of {\em Applied Mathematical Sciences}.
\newblock Springer, New York, second edition, 2012.
\newblock An introduction with applications to partial differential equations.

\bibitem{MR4609068}
J.~Koganemaru and I.~Tice.
\newblock Traveling wave solutions to the inclined or periodic free boundary incompressible {N}avier-{S}tokes equations.
\newblock {\em J. Funct. Anal.}, 285(7):Paper No. 110057, 75, 2023.

\bibitem{MR4785303}
J.~Koganemaru and I.~Tice.
\newblock Traveling wave solutions to the free boundary incompressible {N}avier-{S}tokes equations with {N}avier boundary conditions.
\newblock {\em J. Differential Equations}, 411:381--437, 2024.

\bibitem{MR1307964}
M.~Lanza~de Cristoforis.
\newblock Higher order differentiability properties of the composition and of the inversion operator.
\newblock {\em Indag. Math. (N.S.)}, 5(4):457--482, 1994.

\bibitem{MR3726909}
G.~Leoni.
\newblock {\em A {F}irst {C}ourse in {S}obolev {S}paces}, volume 181 of {\em Graduate Studies in Mathematics}.
\newblock American Mathematical Society, Providence, RI, second edition, 2017.

\bibitem{MR4567945}
G.~Leoni.
\newblock {\em A {F}irst {C}ourse in {F}ractional {S}obolev {S}paces}, volume 229 of {\em Graduate Studies in Mathematics}.
\newblock American Mathematical Society, Providence, RI, 2023.

\bibitem{MR4630597}
G.~Leoni and I.~Tice.
\newblock Traveling wave solutions to the free boundary incompressible {N}avier-{S}tokes equations.
\newblock {\em Comm. Pure Appl. Math.}, 76(10):2474--2576, 2023.

\bibitem{MR3861893}
B.-V. Matioc.
\newblock The {M}uskat problem in two dimensions: equivalence of formulations, well-posedness, and regularity results.
\newblock {\em Anal. PDE}, 12(2):281--332, 2019.

\bibitem{MUSKAT}
M.~Muskat.
\newblock Two fluid systems in porous media. the encroachment of water into an oil sand.
\newblock {\em Phys.}, 5(9):250--264, 1934.

\bibitem{MR1803086}
R.~Narasimhan and Y.~Nievergelt.
\newblock {\em Complex {A}nalysis in {O}ne {V}ariable}.
\newblock Birkh\"auser Boston, Inc., Boston, MA, second edition, 2001.

\bibitem{MR4131404}
H.~Q. Nguyen.
\newblock On well-posedness of the {M}uskat problem with surface tension.
\newblock {\em Adv. Math.}, 374:107344, 35, 2020.

\bibitem{MR4348695}
H.~Q. Nguyen.
\newblock Global solutions for the {M}uskat problem in the scaling invariant {B}esov space {$\dot {B}^1_{\infty,1}$}.
\newblock {\em Adv. Math.}, 394:Paper No. 108122, 28, 2022.

\bibitem{MR4581109}
H.~Q. Nguyen.
\newblock Coercivity of the {D}irichlet-to-{N}eumann operator and applications to the {M}uskat problem.
\newblock {\em Acta Math. Vietnam.}, 48(1):51--62, 2023.

\bibitem{NGCD}
H.~Q. Nguyen.
\newblock Large traveling capillary-gravity waves for {D}arcy flow.
\newblock 2023.
\newblock Preprint, \href{https://arxiv.org/abs/2311.01299}{arXiv:2311.01299}.

\bibitem{MR4090462}
H.~Q. Nguyen and B.~Pausader.
\newblock A paradifferential approach for well-posedness of the {M}uskat problem.
\newblock {\em Arch. Ration. Mech. Anal.}, 237(1):35--100, 2020.

\bibitem{MR4690615}
H.~Q. Nguyen and I.~Tice.
\newblock Traveling wave solutions to the one-phase {M}uskat problem: existence and stability.
\newblock {\em Arch. Ration. Mech. Anal.}, 248(1):Paper No. 5, 58, 2024.

\bibitem{MR1459795}
L.~Simon.
\newblock Schauder estimates by scaling.
\newblock {\em Calc. Var. Partial Differential Equations}, 5(5):391--407, 1997.

\bibitem{MR290095}
E.~M. Stein.
\newblock {\em {S}ingular {I}ntegrals and {D}ifferentiability {P}roperties of {F}unctions}, volume No. 30 of {\em Princeton Mathematical Series}.
\newblock Princeton University Press, Princeton, NJ, 1970.

\bibitem{MR1232192}
E.~M. Stein.
\newblock {\em Harmonic {A}nalysis: {R}eal-{V}ariable {M}ethods, {O}rthogonality, and {O}scillatory {I}ntegrals}, volume~43 of {\em Princeton Mathematical Series}.
\newblock Princeton University Press, Princeton, NJ, 1993.
\newblock With the assistance of Timothy S. Murphy, Monographs in Harmonic Analysis, III.

\bibitem{MR4873829}
N.~Stevenson.
\newblock Periodic gravity-capillary roll wave solutions to the inclined viscous shallow water equations in two dimensions.
\newblock {\em SIAM J. Math. Anal.}, 57(2):1342--1369, 2025.

\bibitem{MR4337506}
N.~Stevenson and I.~Tice.
\newblock Traveling wave solutions to the multilayer free boundary incompressible {N}avier-{S}tokes equations.
\newblock {\em SIAM J. Math. Anal.}, 53(6):6370--6423, 2021.

\bibitem{VCV}
N.~Stevenson and I.~Tice.
\newblock The traveling wave problem for the shallow water equations: well-posedness and the limits of vanishing viscosity and surface tension.
\newblock 2023.
\newblock Preprint, \href{https://arxiv.org/abs/2311.00160}{arXiv:2311.00160}.

\bibitem{COMPRESS}
N.~Stevenson and I.~Tice.
\newblock Well-posedness of the traveling wave problem for the free boundary compressible {N}avier-{S}tokes equations.
\newblock 2023.
\newblock Preprint, \href{https://arxiv.org/abs/2301.00773}{arXiv:2301.00773}.

\bibitem{MR4787851}
N.~Stevenson and I.~Tice.
\newblock Well-posedness of the stationary and slowly traveling wave problems for the free boundary incompressible {N}avier-{S}tokes equations.
\newblock {\em J. Funct. Anal.}, 287(11):Paper No. 110617, 85, 2024.

\bibitem{BORE}
N.~Stevenson and I.~Tice.
\newblock Gravity driven traveling bore wave solutions to the free boundary incompressible {N}avier-{S}tokes equations.
\newblock 2025.
\newblock Preprint, \href{https://arxiv.org/abs/2505.24562}{arXiv:2505.24562}.

\bibitem{SWB}
N.~Stevenson and I.~Tice.
\newblock Stationary wave solutions to two dimensional viscous shallow water equations: theory of small and large solutions.
\newblock 2025.
\newblock Preprint, \href{https://arxiv.org/abs/2502.11899}{arXiv:2502.11899}.

\bibitem{STOKES}
G.~G. Stokes.
\newblock On the theory of oscillatory waves.
\newblock {\em Trans. Cambridge Philos. Soc.}, 8:441--455, 1847.

\bibitem{Strauss_2010}
W.~A. Strauss.
\newblock {Steady water waves}.
\newblock {\em Bull. Amer. Math. Soc. (N.S.)}, 47(4):671--694, 2010.

\bibitem{MR3949722}
W.~A. Strauss and Y.~Wu.
\newblock Rapidly rotating stars.
\newblock {\em Comm. Math. Phys.}, 368(2):701--721, 2019.

\bibitem{MR4135095}
W.~A. Strauss and Y.~Wu.
\newblock Rapidly rotating white dwarfs.
\newblock {\em Nonlinearity}, 33(9):4783--4798, 2020.

\bibitem{MR4557691}
W.~A. Strauss and Y.~Wu.
\newblock Global continuation of a {V}lasov model of rotating galaxies.
\newblock {\em Kinet. Relat. Models}, 16(4):605--623, 2023.

\bibitem{MR1766415}
M.~E. Taylor.
\newblock {\em Tools for {PDE}}, volume~81 of {\em Mathematical Surveys and Monographs}.
\newblock American Mathematical Society, Providence, RI, 2000.
\newblock Pseudodifferential operators, paradifferential operators, and layer potentials.

\bibitem{Toland_1996}
J.~F. Toland.
\newblock {Stokes waves}.
\newblock {\em Topol. Methods Nonlinear Anal.}, 7(1):1--48, 1996.

\bibitem{MR730762}
H.~Triebel.
\newblock {\em Theory of {F}unction {S}paces}, volume~38 of {\em Mathematik und ihre Anwendungen in Physik und Technik}.
\newblock Akademische Verlagsgesellschaft Geest \& Portig K.-G., Leipzig, 1983.

\end{thebibliography}
\end{document}